\documentclass[a4paper]{amsart}
\usepackage{latexsym,bm,stmaryrd}
\usepackage{amsmath,amsthm,amsfonts,amssymb,mathrsfs,pb-diagram}

\usepackage{microtype}
\usepackage{mathtools}
\usepackage{mathbbol,wasysym}

\let\<=\langle
\let\>=\rangle
\def\({\big(}
\def\){\big)}

\def\Z{\mathbb{Z}}
\def\Q{\mathbb{Q}}

\def\N{\mathbb{N}}

\def\lam{\lambda}
\def\Lam{\Lambda}
\def\Sym{\mathfrak{S}}

\newcommand\RR{\mathscr{R}}
\newcommand\R{\RR^{\Lambda}}

\newcommand\HH{\mathscr{H}}

\def\bn[#1,#2]{\begin{bmatrix}#1\\#2\end{bmatrix}}

\def\wnu{\widetilde{\nu}}
\def\fb{\mathbf{b}}
\def\fg{\mathfrak{g}}

\DeclareMathOperator\Hom{Hom}

\DeclareMathOperator\cha{char}

\DeclareMathOperator\id{id}

\DeclareMathOperator\Mod{Mod}
\DeclareMathOperator\rp{Rep}
\DeclareMathOperator\ONH{ONH}
\DeclareMathOperator\skp{SkPol}

\def\rps{\rp_{\rm{super}}}

\def\wk{\tilde{K}}
\def\mcu{\mathcal{U}}
\def\mv{\mathcal{V}}
\def\rd{{\rm{d}}}
\def\rmp{{\rm{p}}}

\def\Mg{\Mod^G}
\def\bk{{\mathbf{k}}}
\def\Ms{\Mod_{\rm{super}}}
\def\Ps{{\rm{Proj}}_{\rm{super}}}
\def\Rs{{\rm{Rep}}_{\rm{super}}}

\begin{document}

\title[]
{graded dimensions and monomial bases for the cyclotomic quiver Hecke superalgebras}
\author{Jun Hu}\address{MIIT Key Laboratory of Mathematical Theory and Computation in Information Security\\
  Beijing Institute of Technology\\
  Beijing, 102488, P.R. China}
\email{junhu404@bit.edu.cn}
\author{Lei Shi}\address{School of Mathematical and Statistics\\
  Beijing Institute of Technology\\
  Beijing, 100081, P.R. China}
\email{3120195738@bit.edu.cn}

\numberwithin{equation}{section}
\newtheorem{prop}[equation]{Proposition}
\newtheorem{thm}[equation]{Theorem}
\newtheorem{thma}[equation]{Theorem A}
\newtheorem{thmb}[equation]{Theorem B}
\newtheorem{thmc}[equation]{Theorem C}
\newtheorem{cor}[equation]{Corollary}
\newtheorem{conj}[equation]{Conjecture}
\newtheorem{hcond}[equation]{Homogeneous Admissible Condition}
\newtheorem{lem}[equation]{Lemma}
\newtheorem{examp}[equation]{Example}
\newtheorem{problem}[equation]{Conjecture}
\newtheorem{ques}[equation]{Question}
\theoremstyle{definition}
\newtheorem{dfn}[equation]{Definition}
\theoremstyle{remark}
\newtheorem{rem}[equation]{Remark}

\begin{abstract} In this paper we derive a closed formula for the $(\Z\times\Z_2)$-graded dimension of the cyclotomic quiver Hecke superalgebra $\RR^\Lam(\beta)$ associated to an {\it arbitrary} Cartan superdatum $(A,P,\Pi,\Pi^\vee)$, polynomials $(Q_{i,j}({\rm x}_1,{\rm x}_2))_{i,j\in I}$, $\beta\in Q_n^+$ and $\Lam\in P^+$. As applications, we obtain a necessary and sufficient condition for which $e(\nu)\neq 0$ in $\RR^\Lam(\beta)$. We construct an explicit monomial basis for the bi-weight space $e(\wnu)\RR^\Lam({\beta})e(\wnu)$, where $\wnu$ is a certain specific $n$-tuple defined in (\ref{wnu0}). In particular, this gives rise to a monomial basis for the cyclotomic odd nilHecke algebra. Finally, we consider the case when $\beta=\alpha_{1}+\alpha_{2}+\cdots+\alpha_{n}$ with $\alpha_1,\cdots,\alpha_n$ distinct. We construct an explicit monomial basis of $\RR^\Lam(\beta)$ and show that it is indecomposable in this case.
\end{abstract}

\maketitle

\section{Introduction}

The quiver Hecke algebras (or, KLR algebras) and their cyclotomic quotients were introduced in the work of Khovanov-Lauda (\cite{KL1}, \cite{KL2}) and of Rouquier (\cite{Rou1}). They play an important role in the categorification of quantum groups and their integrable highest weight modules (\cite{KK}). In recent years there have been many remarkable applications of these algebras in the modular representation theory of symmetric groups and Hecke algebras, low-dimensional topology and other areas, see \cite{BK:GradedKL}, \cite{HM}, \cite{HuMathas:SeminormalQuiver}, \cite{Kle2}, \cite{Rou2}, \cite{SVV}, \cite{VV}, \cite{Web} and the references therein.

In order to generalise the KLR's construction to the super case, several new families of algebras including the quiver Hecke superalgebras and quiver Hecke Clifford algebras were introduced in \cite{KKT}. To define the quiver Hecke superalgebras, one has to start with a generalised Cartan matrix $A$ (\cite{Kac}) parameterized by an index set $I=I_{\rm{even}}\sqcup I_{\rm{odd}}$ subject to some natural conditions. Then one has to manipulate a mixture of the usual nilHecke algebra and the odd nilHecke algebra in \cite{EKL} (or, the spin nilHecke algebra initially studied in \cite{W} and \cite{KW}). That says, to each $i\in I_{\rm{even}}$ is attached the usual nilHecke algebra, while to each $i\in I_{\rm{odd}}$ is attached the odd nilHecke algebra. If $I_{\rm{odd}}=\emptyset$, then the construction in \cite{KKT} reduces to the original KLR construction. The resulting quiver Hecke superalgebras and their cyclotomic quotients are both $(\Z\times\Z_2)$-graded algebras.

In \cite{HS} the authors used Kang-Kashiwara's categorification of integral highest weight modules over quantum groups (\cite{Lu}) to derive a closed formula for the $\Z$-graded dimension of the usual quiver Hecke algebras. In the current paper, we generalise this formula to the super case to derive a closed formula for the $(\Z\times\Z_2)$-graded dimension of the cyclotomic quiver Hecke superalgebra. To state our main results, we need some notations.

Let $\bigl({\rm{A}}=(a_{ij})_{i,j\in I},P,\Pi,\Pi^\vee\bigr)$ be a Cartan superdatum in the sense of \cite[\S4.1]{KKO2}. Let $x,q$ be two indeterminates over $\Z$. Consider the quotient ring $\Z[x]/\<x^4-1\>$. We define $$
\pi:=x^2+\<x^4-1\>,\quad\, \sqrt{\pi}:=x+\<x^4-1\> .
$$
Then $\Z[x]/\<x^4-1\>=\Z[\sqrt{\pi}]$. For any ring $R$, we set $R^\pi:=R\otimes_{\Z}\Z[\pi]$.

%Let $\Ms(\R(\beta))$ be the category of arbitrary $(\Z\times\Z_2)$-graded $\R(\beta)$-modules. Let $\Ps(\R(\beta))$ and $\Rs\R(\beta)$ be the full subcategory of $\Ms(\R(\beta))$ consisting of projective $\R(\beta)$-supermodules finite dimensional over $\bk$ and $\R(\beta)$-supermodules finite dimensional over $\bk$ respectively.

Let $\RR^\Lam(\beta)$ be the cyclotomic quiver Hecke superalgebra over a field $\mathbf{k}$ associated to the Cartan superdatum $\bigl({\rm{A}}=(a_{ij})_{i,j\in I},P,\Pi,\Pi^\vee\bigr)$,  a family of polynomials $\{Q_{ij}({\rm x}_1,{\rm x}_2)|i,j\in I\}$, $\beta\in Q_n^+$ and $\Lam\in P^+$ as defined in \cite{KKT}. Kang, Kashiwara and Oh (\cite{KKO1}, \cite{KKO2}) studied the supercategorification of quantum Kac-Moody algebras via the cyclotomic quiver Hecke superalgebra (see \cite{HW} for a closely related work). To this end they have introduced in \cite{KKO2} several families of quantum superalgebras (also see \cite{BKM}). Among them the algebra $\mathcal{U}(\fg)$ over $\Q(q)^\pi$ corresponding to the Cartan superdatum $\bigl({\rm{A}},P,\Pi,\Pi^\vee\bigr)$ (introduced in \cite{KKO2}) is directly related with the supercategorification theorem via the cyclotomic quiver Hecke superalgebras $\RR^\Lam(\beta)$. More precisely, let $\mv(\Lam)$ be the $P$-weighted $\mathcal{U}_q(\mathfrak{g})$-module generated by its vector $v_\Lam$ of weight $\Lam$ as defined in \cite[(5.4)]{KKO2}. Let $\mv_{\mathbb{A}^\pi}(\Lam):=\mcu_{\mathbb{A}^\pi}(\fg)v_\Lam$ be the $\mathbb{A}^\pi$-form of $\mv(\Lam)$, where $\mathbb{A}^\pi:=\Z[q,q^{-1}]^\pi$. They showed in \cite[Theorem 8.14]{KKO2} that the category $\Ps(\R(\beta))$ of $(\Z\times\Z_2)$-graded projective $\R(\beta)$-modules finitely dimensional over $\bk$ gives a supercategorification of $\mv_{\mathbb{A}^\pi}(\Lam)$. In particular, there is a $\mcu_{\mathbb{A}^\pi}(\fg)$-module isomorphism between the Grothendieck group $[\Ps(\RR^\Lam)]$ and $\mv_{\mathbb{A}^\pi}(\Lam)$.

Let $\Rs\R(\beta)$ be the category of $(\Z\times\Z_2)$-graded $\R(\beta)$-modules finitely dimensional over $\bk$.

\begin{dfn} For any $M=\oplus_{a\in\Z}(M_{a,\overline{0}}\oplus M_{a,\overline{1}})\in\rps(\R(\beta))$, we define its $(q,\pi)$-dimension as: $$
\dim_q^{\pi}(M):=\sum_{a\in\Z}\bigl(\dim_{\bk}M_{a,\overline{0}}+\pi\dim_{\bk}M_{a,\overline{1}}\bigr)q^a\in\Z[q^{\pm1}]^\pi .
$$
\end{dfn}

The following theorem, which generalize \cite[Theorem 1.1]{HS} in the non-super case, is the first main result of this paper, where we refer the readers to Section 2, (\ref{Gauss2}), Definition \ref{keydfn1} and (\ref{pdef}) for unexplained notations used here.

\begin{thm}\label{mainthmA} Let $\Lam\in P^+$ and $\beta\in Q_n^+$. Let $\nu=(\nu_1,\cdots,\nu_n),\nu'=(\nu'_1,\cdots,\nu'_n)\in I^\beta$. Then $$
\dim_q^{\pi}e(\nu)\RR^\Lam(\beta)e(\nu')=\sum_{\substack{w\in\Sym(\nu,\nu')}}\prod_{t=1}^{n}\Bigl([N^{\Lam}(w,\nu,t)]^\pi_{\nu_t}
q_{\nu_t}^{N^{\Lam}(1,\nu,t)-1}\pi_{\nu_t}^{\rmp(w,\nu,t)}\Bigr).
$$
\end{thm}
The proof of the above theorem is based on Kang-Kashiwara-Oh's supercategorification (\cite[\S8.4]{KKO2}) of $\mv_{\mathbb{A}^\pi}(\Lam)$ via the cyclotomic quiver Hecke superalgebras $\RR^\Lam(\beta)$.
In particular, the above theorem implies that the $\Z$-graded dimension and the (ungraded) dimension of $\RR^\Lam(\beta)$ do not depend on the choices of the decomposition $I=I_{\rm{even}}\sqcup I_{\rm{odd}}$. As a first application of Theorem \ref{mainthmA},
we obtain the following criterion for $e(\nu)\neq 0$ in $\RR^\Lam(\beta)$ which is valid for {\it arbitrary} Cartan superdatum.

\begin{cor}\label{maincorA1} Let $\Lam\in P^+$, $\beta\in Q^+$ and $\nu=(\nu_1,\cdots,\nu_n)\in I^\beta$. Then $e(\nu)\neq 0$ in $\RR^\Lam(\beta)$ if and only if $$
\sum\limits_{w\in\Sym(\nu,\nu)}\prod\limits_{t=1}^{n}N^\Lam(w,\nu,t)\neq 0 .
$$
\end{cor}

Our second application of Theorem \ref{mainthmA} is about the monomial bases of $\RR^\Lam(\beta)$. We fix $p\in\N$, $\fb:=(b_1,\cdots,b_p)\in\N^p$ and $\nu^1,\cdots,\nu^p\in I$ such that $\nu^i\neq\nu^j$ for any $1\leq i\neq j\leq p$ and $\sum_{i=1}^{p}b_i=n$. We define \begin{equation}\label{wnu0}
\wnu=(\wnu_1,\cdots,\wnu_n):=\bigl(\underbrace{\nu^1,\cdots,\nu^1}_{\text{$b_1$ copies}},\cdots,\underbrace{\nu^p,\cdots,\nu^p}_{\text{$b_p$ copies}}\bigr)\in I^\beta ,
\end{equation}
where $\beta=\sum_{i=1}^{p}b_i\alpha_{\nu^i}\in Q_n^+$. Set $b_0:=0, c_t:=\sum_{i=0}^{t}b_i$ for any $0\leq t\leq p$. The following theorem is the second main result of this paper, where we refer the readers to (\ref{sb}) and (\ref{nwvk}) for the definitions of $\Sym_{\fb}$ and $N^\Lam(\wnu,k)$.

\begin{thm}\label{mainthmB} Let $\Lam\in P^+$. Let $\wnu$ be given as in (\ref{wnu0}). Let $\beta\in Q_n^+$ such that $\wnu\in I^\beta$.
The following set \begin{equation}\label{base1B}
\Bigl\{e(\wnu)\prod_{k=1}^{n}x_{k}^{r_{k}}\tau_{w }\Bigm|\begin{matrix}\text{$w\in\Sym_{\fb}$, for any $1\leq i\leq p$, $c_{i-1}<k\leq c_{i}$,}\\
\text{$0\leq r_{k}<N^\Lam(\wnu,k)$}\end{matrix}\Bigr\}
\end{equation} forms a $K$-basis of $e(\wnu)\RR^\Lam({\beta})e(\wnu)$.
\end{thm}

We remark that the above theorem is a {\it non-trivial} generalization of the corresponding result \cite[Theorem 5.8]{HS} for the non-super case. This is because the original argument in the proof of \cite[Theorem 5.8]{HS} actually does not work in the super case so we have to adopt a completely different approach to prove Theorem \ref{mainthmB}. Also due to the complexity of its super structure, we are currently unable to generalise \cite[Theorem 1.5]{HS} in its full generality to the super case.

The bi-weight space $e(\wnu)\RR^\Lam({\beta})e(\wnu)$ which we considered in Theorem \ref{mainthmB} contains the following special case: $$
\text{$p=1$ and $\nu^1\in I_{\rm{odd}}$, i.e., the cyclotomic odd nilHecke algebra case.}
$$
Thus Theorem \ref{mainthmB} yields a monomial basis for the cyclotomic odd nilHecke algebra $\ONH_n^\ell$. That is,

\begin{cor}\label{maincor1} Let $\ell,n\in\Z_{\geq 0}$. \begin{enumerate}
\item[1)] Let $\ell,n\in\Z_{\geq 0}$. Then $\ONH_n^\ell\neq 0$ if and only if $\ell\geq n$;
\item[2)] Assume $\ell\geq n$. Then the following set \begin{equation}\label{bases1C}
\Sigma:=\{x_{1}^{k_1}\cdots x_{n}^{k_n}\tau_w\,|\,w\in\Sym_n,\,\,0\leq k_i\leq\ell-i,\forall\,1\leq i\leq n\}
\end{equation} forms a $\bk$-basis of $\ONH_n^\ell$.
\end{enumerate}
\end{cor}

The readers can find that the above corollary is an analogue of the corresponding result \cite[Theorem 2.34]{HuL} for the usual cyclotomic nilHecke algebra of type $A$. However, the original argument does not transfer to the
cyclotomic odd nilHecke case. In fact, we shall give a self-contained new argument for the proof of Corollary \ref{maincor1} (which also works for the original non-super case).

As a third application of our graded dimension formula in Theorem \ref{mainthmA}, we study the algebra $\RR^\Lam(\beta)$ in the case when $\beta=\alpha_{1}+\alpha_{2}+\cdots+\alpha_{n}$ with $\alpha_1,\cdots,\alpha_n$ distinct. We show that in this case $\R(\beta)$ is indecomposable and we construct an explicit monomial basis for it, which gives the third main result of this paper.

\begin{thm}\label{mainthmC}  Assume $\beta=\alpha_{1}+\alpha_{2}+\cdots+\alpha_{n}$, $\alpha_i\neq\alpha_j,\,\forall\,1\leq i\neq j\leq n$. \begin{enumerate}
\item[1)]  Let $\mu,\nu\in I^\beta$ such that $e(\nu)\RR^\Lam(\beta)e(\mu)\neq 0$. Then the elements in the following set $$
 \Bigl\{\tau_{d_{\mu,\nu}}\prod_{k=1}^{n}x_{k}^{r_{k}}e(\mu)\Bigm|\, 0\leq r_{k}<N^\Lam(d_{\mu,\nu},\mu,k)\Bigr\}.
$$
form a $\bk$-linear basis of $e(\nu)\RR^{\Lambda}(\beta)e(\mu)$, where $d_{\mu,\nu}$ is the unique element in $\Sym_n$ such that $d_{\mu,\nu}\mu=\nu$;
\item[2)] $\RR^\Lam(\beta)$ is indecomposable.\end{enumerate}
\end{thm}

The content of the paper is organised as follows. In Section 2, we give some preliminary definitions and notations for Cartan superdatum, the algebra $\mathcal{U}(\fg)$ and the quiver Hecke superalgebra $\R(\beta)$.
In Section 3, after recalling Kang-Kashiwara-Oh's supercategorification of $\mv_{\mathbb{A}^\pi}(\Lam)$ via the cyclotomic quiver Hecke superalgebras $\RR^\Lam(\beta)$, we give the proof of the first main result Theorem \ref{mainthmA} and its Corollary \ref{maincorA1}. In Section 4, we first give a self-contained proof of the monomial basis result Corollary \ref{maincor1} for the cyclotomic odd nilHecke algebra $\ONH_n^\ell$. Then we generalise this result to give a monomial basis for the bi-weight space  $e(\wnu)\RR^\Lam({\beta})e(\wnu)$, i.e., give a proof of our second main result Theorem \ref{mainthmB}. In Section 5, we apply Theorem \ref{mainthmA} to give a proof of our third main result Theorem \ref{mainthmC}.

\bigskip
\centerline{Acknowledgements}
\bigskip

The research was supported by the National Natural Science Foundation of China (No. 12171029).
\bigskip

\section{Preliminary}

Let $I$ be an indexing set. An integral matrix $(a_{ij})_{i,j\in I}$ is called a Cartan matrix if it satisfies: i) $a_{ii}=2$, ii) $a_{ij}\leq 0$ for $i\neq j$, iii) $a_{ij}=0$ if and only if $a_{ji}=0$. We say ${\rm{A}}$ is symmetrizable if there is a diagonal matrix ${\rm{D}}={\rm{diag}}(\rd_i\in\Z_{>0}|i\in I)$ such that ${\rm{DA}}$ is symmetric.

Let $\bigl({\rm{A}}=(a_{ij})_{i,j\in I},P,\Pi,\Pi^\vee\bigr)$ be a Cartan superdatum in the sense of \cite[\S4.1]{KKO2}. That means, \begin{enumerate}
\item[CS1)] ${\rm{A}}$ is a symmetrizable Cartan matrix;
\item[CS2)] $P$ is a free abelian group, which is called the weight lattice;
\item[CS3)] $\Pi=\{\alpha_i\in P|i\in I\}$, called the set of simple roots, is $\Z$-linearly independent;
\item[CS4)] $\Pi^\vee=\{h_i\in P|i\in I\}\subset P^\vee=\Hom_\Z(P,\Z)$, called the set of simple coroots, satisfies that $\<h_i,\alpha_j\>=a_{ij}$ for all $i,j\in I$;
\item[CS5)] there is a decomposition $I=I_{\rm{even}}\sqcup I_{\rm{odd}}$ such that \begin{equation}\label{evenodd}
a_{ij}\in 2\Z, \quad \text{for all $i\in I_{\rm{odd}}$ and $j\in I$.}
\end{equation}
\end{enumerate}
The diagonal matrix ${\rm{D}}$ gives rise to a symmetric bilinear form $(-|-)$ on $P$ which satisfies: $$
(\alpha_i|\lam)=\rd_i\<h_i,\lam\>\quad \text{for all $\lam\in P$.}
$$
In particular, we have $(\alpha_i|\alpha_j)=\rd_i a_{ij}$ and hence $\rd_i=(\alpha_i|\alpha_i)/2$ for each $i\in I$.

We define the root lattice $Q$ to be the abelian group $\oplus_{i\in I}\Z\alpha_i$. We call $Q^+:=\oplus_{i\in I}\Z_{\geq 0}\alpha_i$ the positive root lattice. For any $\alpha=\sum_{i\in I}k_i\alpha_i\in Q^+$, we define $|\alpha|:=\sum_{i\in I}k_i$. For any $n\in\Z_{\geq 0}$, we define $Q_n^+:=\{\alpha\in Q^+||\alpha|=n\}$. Let $P^+:=\{\lam\in P|\text{$\<h_i,\lam\>\in\Z_{\geq 0}$ for all $i\in I$}\}$. Any element $\lam\in P^+$ is called a dominant integral weight.

For a Cartan superdatum $({\rm{A}},P,\Pi,\Pi^\vee)$, we define the parity function $\rmp: I\rightarrow\{\overline{0},\overline{1}\}$ by \begin{equation}\label{pi1}
\rmp(i):=\begin{cases} \overline{1}, &\text{if $i\in I_{\rm{odd}}$;}\\
\overline{0}, &\text{if $i\in I_{\rm{even}}$.}
\end{cases}
\end{equation}

%Let $x$ be an indeterminate over $\Z$. Consider the quotient ring $\Z[x]/\<x^4-1\>$. We define $$
%\pi:=x^2+\<x^2-1\>,\quad\, \sqrt{\pi}:=x+\<x^2-1\> .
%$$
%Thus we have $\Z[x]/\<x^4-1\>=\Z[\sqrt{\pi}]$ and we use $\Z[\pi]$ to denote its subring generated by $1$ and $\pi$.

Let $q$ be an indeterminate over $\Z$. For each $i\in I$, we define \begin{equation}\label{Gauss2}
\pi_i:=\pi^{\rmp(i)},\,\,q_i:=q^{\rd_i},\,\,[n]_i^{\pi}:=\frac{(\pi_iq_i)^n-q_i^{-n}}{\pi_iq_i-q_i^{-1}},\,\,
[n]_i^{\pi}!:=\prod_{k=1}^{n}[k]_i^{\pi},\,\,\text{for $n\in\Z_{\geq 0}$}.
\end{equation}
In particular, the assumption (\ref{evenodd}) implies that $\pi_i^{a_{ij}}=1$ for any $i,j\in I$. For any ring $R$, we define $R^\pi:=R\otimes_{\Z}\Z[\pi]$.

\begin{dfn}[{\cite[Definition 5.1]{KKO2}}] \label{ug}Let $({\rm{A}},P,\Pi,\Pi^\vee)$ be a Cartan superdatum. Let $\mcu(\fg)$ be the unital associative $\Q(q)^\pi$-algebra with generators $e_i, f_i$ and $\wk_i^{\pm 1}$ ($i\in I$) and the following defining relations: $$\begin{aligned}
& \wk_i\wk_j=\wk_j\wk_i,\,\, \wk_ie_j\wk_i^{-1}=q_i^{2a_{ij}}e_j,\,\, \wk_if_j\wk_i^{-1}=q_i^{-2a_{ij}}f_j,\\
& e_if_j-\pi^{\rmp(i)\rmp(j)}q_i^{-a_{ij}}f_je_i=\delta_{ij}\frac{1-\wk_i}{1-q_i^2\pi_i}\,\,\,\,(i,j\in I),\\
&\sum_{k=0}^{1-a_{ij}}\bigl(-\pi^{\rmp(i)\rmp(j)}\bigr)^k\pi_i^{k(k-1)/2}f_i^{\{1-a_{ij}-k\}}f_jf_i^{\{k\}}=0\,\,\,\, (i\neq j),\\
&\sum_{k=0}^{1-a_{ij}}\bigl(-\pi^{\rmp(i)\rmp(j)}\bigr)^k\pi_i^{k(k-1)/2}e_i^{\{1-a_{ij}-k\}}e_je_i^{\{k\}}=0\,\,\,\, (i\neq j),
\end{aligned}
$$
where $f_i^{\{n\}}=f_i^n/[n]_i^{\pi}!$ and $e_i^{\{n\}}=e_i^n/[n]_i^{\pi}!$.
\end{dfn}

The algebra $\mcu(\fg)$ has an anti-involution $\tau$ given by $$
e_i\mapsto f_i,\quad f_i\mapsto e_i,\quad \tilde{K}_i\mapsto\tilde{K}_i,\,\,\forall\,i\in I .
$$
Set $$
{\mathbb{k}}:=\Q(q)^\pi,\,\,\mathbb{A}^\pi:=\Z[q,q^{-1}]^\pi .
$$
Following \cite[\S5]{KKO2}, we define the $\mathbb{A}^\pi$-form $\mcu_{\mathbb{A}^\pi}(\fg)$ of $\mcu(\fg)$ to be the $\mathbb{A}^\pi$-subalgebra of $\mcu(\fg)$ generated by the elements $e_i^{\{n\}}, f_i^{\{n\}},\tilde{K}_i^{\pm 1}$ for $i\in I$ and $n\in\Z_{\geq 0}$. We denote by $\mcu_{\mathbb{A}^\pi}^+(\fg)$ (resp., $\mcu_{\mathbb{A}^\pi}^{-}(\fg)$) the $\mathbb{A}^\pi$-subalgebra of $\mcu_{\mathbb{A}^\pi}(\fg)$ generated by the elements $e_i^{\{n\}}, i\in I, n\in\Z_{\geq 0}$ (resp., by the elements $f_i^{\{n\}}, i\in I, n\in\Z_{\geq 0}$).

It is clear that both $\mcu(\fg)$ and $\mcu_{\mathbb{A}^\pi}(\fg)$ are $Q$-graded such that $$
\deg e_i=\alpha_i,\,\,\,\deg f_i=-\alpha_i,\,\,\,\deg\wk_i=0,\,\,\forall\,i\in I .
$$
For each $\alpha\in Q$, we use $\mcu(\fg)_\alpha$ to denote the corresponding homogeneous component of $\mcu(\fg)$.

\begin{dfn}[{\cite[\S3, (5.1)]{KKO2}}] Let $G$ be a subset of $P$ such that $G+Q\subset G$. A $\mcu(\fg)$-module $V$ is called a $G$-weighted module if $V=\oplus_{\mu\in G}V_\mu$ such that $$
\mcu(\fg)_\alpha V_\mu\subseteq V_{\mu+\alpha},\,\,\wk_i|_{V_\mu}=q_i^{2\<h_i,\mu\>}\pi_i^{\<h_i,\mu\>}\id_{V_\mu} .
$$
\end{dfn}
We denote by $\Mg(\mcu(\fg))$ the category of $G$-weighted $\mcu(\fg)$-modules.
%
%We set $$
%P_{\rm{even}}:=\{\lam\in P|\<h_i,\lam\>\in 2\Z, \forall\,i\in I_{\rm{odd}}\},\quad P_{\rm{even}}^+:=P^+\cap P_{\rm{even}}.
%$$

Throughout this paper, let $\bk$ be a field of characteristic different from $2$. Let $\bigl({\rm{A}},P,\Pi,\Pi^\vee\bigr)$ be a Cartan superdatum. Let ${\rm x}_1,\cdots,{\rm x}_n$ be $n$ indeterminates over $\bk$. For any $n\geq 2$ and $\nu\in I^n$, set $$
\mathcal{P}_\nu:=\bk\<{\rm x}_1,\cdots,{\rm x}_n\>/\<{\rm x}_a{\rm x}_b-(-1)^{\rmp(\nu_a)\rmp(\nu_b)}{\rm x}_b{\rm x}_a|1\leq a<b\leq n\> .
$$
Then $\mathcal{P}_\nu$ is a superalgebra if we endow the image of each ${\rm x}_k$ the parity $\rmp(\nu_k)$. We refer the readers to \cite[\S12]{Klesh:book} for general theory of superalgebras. For $i,j\in I$, we choose an element $Q_{ij}\in\mathcal{P}_{(ij)}$ which is of the form $$
Q_{ij}({\rm x}_1,{\rm x}_2)=\sum_{r,s\geq 0}t_{i,j;(r,s)}{\rm x}_1^r{\rm x}_2^s,
$$
where the coefficient satisfies that \begin{enumerate}
\item $t_{i,j;(r,s)}\neq 0$ only if $-2(\alpha_i|\alpha_j)-r(\alpha_i|\alpha_i)-s(\alpha_j|\alpha_j)=0$;
\item $t_{i,j;(r,s)}=t_{j,i;(s,r)}$, $t_{i,j;(-a_{ij},0)}\in\bk^\times$;
\item $t_{i,j;(r,s)}=0$ if either $i=j$ or $i\in I_{\rm{odd}}$ and $r$ is odd.
\end{enumerate}
In the definition of the quiver Hecke superalgebras and their cyclotomic quotients given below, we shall only use the element $Q_{i,j}(x_a,x_b)e(\nu)$ in the case when $\nu_a=i$ and $\nu_b=j$.
For any $\beta\in Q_n^+$, we define $$
I^\beta:=\{\nu=(\nu_1,\cdots,\nu_n)\in I^n|\sum_{s=1}^{n}\alpha_{\nu_s}=\beta\}. $$

\begin{dfn}\label{qhs} Let $({\rm{A}},P,\Pi,\Pi^\vee)$ be a Cartan superdatum, $\beta\in Q_n^+$ and $\{Q_{i,j}|i,j\in I\}$ be chosen as above. The associated degree $n$ quiver Hecke superalgebras $\RR(\beta)$ is the superalgebra over $\bk$ (with the identity element $e(\beta)$), which is defined by the generators $$
e(\nu)\, (\nu\in I^\beta), x_k\, (1\leq k\leq n),\, \tau_a (1\leq a\leq n-1),
$$
the parity $$
\rmp(e(\nu))=0,\quad \rmp(x_ke(\nu))=\rmp(\nu_k),\quad \rmp(\tau_ae(\nu))=\rmp(\nu_a)\rmp(\nu_{a+1}),
$$
and the following relations: $$\begin{aligned}
& e(\mu)e(\nu)=\delta_{\mu,\nu}e(\nu),\,\,\text{for $\mu,\nu\in I^\beta$}, \,\,e(\beta)=\sum_{\nu\in I^\beta}e(\nu),\\
& x_px_qe(\nu)=(-1)^{\rmp(\nu_p)\rmp(\nu_q)}x_qx_pe(\nu),\,\,\text{if $p\neq q$,}\\
&x_pe(\nu)=e(\nu)x_p,\,\,\,\tau_ae(\nu)=e(s_a\nu)\tau_a,\,\,\text{where $s_a=(a,a+1)$,}\\
& \tau_ax_pe(\nu)=(-1)^{\rmp(\nu_p)\rmp(\nu_a)\rmp(\nu_{a+1})}x_p\tau_ae(\nu),\,\,\text{if $p\neq a,a+1$,}\\
& \bigl(\tau_ax_{a+1}-(-1)^{\rmp(\nu_a)\rmp(\nu_{a+1})}x_a\tau_a\bigr)e(\nu)\\
&\qquad =\bigl(x_{a+1}\tau_a-(-1)^{\rmp(\nu_a)\rmp(\nu_{a+1})}\tau_ax_a\bigr)=\delta_{\nu_a,\nu_{a+1}}e(\nu),\\
&\tau_a^2e(\nu)=Q_{\nu_a,\nu_{a+1}}(x_a,x_{a+1})e(\nu),\\
&\tau_a\tau_be(\nu)=(-1)^{\rmp(\nu_a)\rmp(\nu_{a+1})\rmp(\nu_b)\rmp(\nu_{b+1})}\tau_b\tau_ae(\nu),\,\,\text{if $|a-b|>1$},\\
& (\tau_{a+1}\tau_a\tau_{a+1}-\tau_a\tau_{a+1}\tau_a)e(\nu)\\
&=\begin{cases}
\frac{Q_{\nu_a,\nu_{a+1}}(x_{a+2},x_{a+1})-Q_{\nu_a,\nu_{a+1}}(x_{a},x_{a+1})}{x_{a+2}-x_a}e(\nu), &\text{if $\nu_a=\nu_{a+2}\in I_{\rm{even}}$;}\\
(-1)^{\rmp(\nu_{a+1})}(x_{a+2}-x_a)\frac{Q_{\nu_a,\nu_{a+1}}(x_{a+2},x_{a+1})-Q_{\nu_a,\nu_{a+1}}(x_{a},x_{a+1})}{x_{a+2}^2-x_a^2}e(\nu), &\text{if $\nu_a=\nu_{a+2}\in I_{\rm{odd}}$;}\\
0, &\text{otherwise.}
\end{cases}
\end{aligned}
$$
\end{dfn}

\begin{prop}[{\cite[Corollary 3.15]{KKT}}]\label{stdBasis} Let $\beta\in Q^+_n$. For each $w\in\Sym_n$, we fix a reduced expression $w=s_{i_1}\cdots s_{i_l}$, and define $\tau_w:=\tau_{i_1}\cdots\tau_{i_l}$, then the following set $$\{e(\nu)x_1^{t_1}\cdots x_n^{t_n}\tau_w|t_i\in \Z_{\geq 0},\,i=1,\cdots n,\,w\in \Sym_n,\,\nu\in I^\beta\}
$$ forms a basis of the free $\bk$-module $\RR(\beta)$.
\end{prop}

%Let $u$ be an indeterminate and $\beta\in Q_n^+$. For each $i\in I$ and $k\geq 0$, we take $c_{i;k}\in\bk_{k(\alpha_i|\alpha_i)}$ such that $c_{i;0}=1$ and $c_{i;k}=0$ if $i\in I_{\rm{odd}}$ and $k$ is odd.
If $\Lam\in P^+,\,i\in I$ and $u$ is an indeterminate over $\Z$, then we define $$
a_i^\Lam(u)=u^{\<h_i,\Lam\>},\quad
a^\Lam(x_1):=\sum_{\nu\in I^\beta}x_1^{\<h_{\nu_1},\Lam\>}e(\nu) .
$$

\begin{dfn} Let $\beta\in Q_n^+$, $\Lam\in P^+$. The cyclotomic quiver Hecke superalgebra $\RR^\Lam(\beta)$ is defined to be the quotient algebra: $$
\RR^\Lam(\beta):=\RR(\beta)/\<a^\Lam(x_1)\> .
$$
\end{dfn}

Both the algebra $\RR(\beta)$ and $\RR^\Lam(\beta)$ are $\Z$-graded by setting $$
\deg_\Z(e(\nu))=0,\quad \deg_\Z(x_ke(\nu))=(\alpha_{\nu_k}|\alpha_{\nu_k}),\quad \deg_\Z(\tau_ae(\nu))=-(\alpha_{\nu_a}|\alpha_{\nu_{a+1}}).
$$
Similarly, $\RR^\Lam(\beta)$ inherits a $\Z_2$-grading from $\RR(\beta)$. That says, $\RR^\Lam(\beta)$ is a superalgebra too.

\begin{rem} By some abuse of notations, we shall use the same symbols to denote the generators of both $\RR(\beta)$ and $\RR^\Lam(\beta)$. By Proposition \ref{stdBasis}, for any $\nu\in I^\beta$, the $\bk$-subalgebra of $\RR(\beta)$ generated by $x_1e(\nu),\cdots,x_ne(\nu)$ is canonically isomorphic to $\mathcal{P}_\nu$ via the correspondence $x_ie(\nu)\mapsto{\rm x}_i$, $\forall\,1\leq i\leq n$. There are natural left (resp., right) actions of $\mathcal{P}_\nu$ on $e(\nu)\RR(\beta)$ (resp., $\RR(\beta)e(\nu)$) which are defined by multiplication followed with substituting each ${\rm x}_j$ with $x_j$. Similarly, there are natural left (resp., right) actions of $\mathcal{P}_\nu$ on $e(\nu)\RR^\Lam(\beta)$ (resp., $\RR^\Lam(\beta)e(\nu)$).
\end{rem}

\section{$(\Z\times\Z_2)$-graded dimensions}

Let $\Ms(\R(\beta))$ be the category of arbitrary $(\Z\times\Z_2)$-graded $\R(\beta)$-modules.\footnote{The notation $\Ms(\R(\beta))$ was used in \cite{KKO2} to denote the category of $\Z$-graded $\R(\beta)$-supermodules. We believe that they indeed mean $(\Z\times\Z_2)$-graded $\R(\beta)$-modules as otherwise they can not define the $(q,\pi)$-dimension in \cite[(8.8)]{KKO2} for any module in $\Rs(\R(\beta))$.} Let $\Ps(\R(\beta))$ and $\Rs(\R(\beta))$ be the full subcategory of $\Ms(\R(\beta))$ consisting of projective $\R(\beta)$-supermodules finitely dimensional over $\bk$ and $\R(\beta)$-supermodules finitely dimensional over $\bk$.

Let $q$ be the grading shift functor on $\Ms(\R(\beta))$. That means, $$(qM)_j=M_{j-1}$$ for any $M=\oplus_{\substack{j\in \Z}}M_j\in \Ms(\R(\beta))$. Let
$$\begin{aligned}\phi: \R(\beta)&\rightarrow\R(\beta)\\
a&\mapsto (-1)^ia, \quad \forall\,a\in\R(\beta)_i, \end{aligned}$$ be the parity involution of $\R(\beta)$.

Let $\Pi: \Ms(\R(\beta))\rightarrow\Ms(\R(\beta))$ be the parity changing functor. Then $$\begin{aligned}
\Pi(M)=\{\pi(x)|x\in M\},\quad \pi(x)+\pi(x')=\pi(x+x'),\\
a\pi(x):=\pi(\phi(a)x),\quad \forall\,a\in\R(\beta), x,x'\in M .\\
(\Pi M)_{i}:=\{\pi(x)|x\in M_{1-i}\},\,\,\,\forall\,i\in \Z_2 .
\end{aligned}
$$

%We also have the parity shifting functor $\Pi$, i.e.$$\begin{aligned}&(\Pi M)_{i}:=\{\pi(x)|x\in M_{1-i}\}\,\,(i\in \Z_2)\\
%&a\cdot\pi(x):=\pi((-1)^{i}(a)x)\,(p(a)=i,\,x\in M)
%\end{aligned}
%$$

For any $\beta\in Q_n^+$ and $i\in I$, we set $$
e(\beta,i):=\sum_{\nu=(\nu_1,\cdots,\nu_n)\in I^{\beta}}e(\nu_1,\cdots,\nu_n,i).
$$
Kang, Kashiwara and Oh have introduced restriction functors and induction functors in \cite{KKO2} as follows: $$\begin{aligned}
E_i^\Lam:\, \Ms(\RR^\Lam(\beta+\alpha_i))&\rightarrow \Ms(\RR^\Lam(\beta)),\\
N&\mapsto e(\beta,i)N=e(\beta,i)\RR^\Lam(\beta+\alpha_i)\otimes_{\RR^{\Lam}(\beta+\alpha_i)}N,\\
F_i^\Lam:\, \Ms(\RR^\Lam(\beta))&\rightarrow \Ms(\RR^\Lam(\beta+\alpha_i)),\\
M&\mapsto \RR^\Lam(\beta+\alpha_i)e(\beta,i)\otimes_{\RR^{\Lam}(\beta)}M .
\end{aligned}
$$
For any category $\mathcal{C}$ we use $[\mathcal{C}]$ to denote its Grothendieck group. Then the functors $q$ and $\Pi$ can descend to the Grothendieck group $[\Ms(\R(\beta))]$ for which we denote by $q[M]=[qM]$ and $\pi[M]=[\Pi M]$. This makes $[\Ms(\R(\beta))]$ an $\mathbb{A}^\pi$-module, where $\mathbb{A}=\Z[q,q^{-1}]$.  Let ${\rm E}_i:=[E_i^\Lam]$, ${\rm F}_i:=[F_i^\Lam]$, where $[E_i^\Lam]: [\Ps(\RR^\Lam(\beta+\alpha_i))]\rightarrow [\Ps(\RR^\Lam(\beta))]$ and $[F_i^\Lam]: [\Ps(\RR^\Lam(\beta))]\rightarrow [\Ps\RR^\Lam(\beta+\alpha_i))]$ are the naturally induced map on the Grothendieck groups.
We define $$
\RR^\Lam=\bigoplus_{n\geq 0, \beta\in Q_n^+}\RR^\Lam(\beta) .
$$
Then $$\begin{aligned}
[\Ps(\RR^\Lam)] & =\oplus_{\beta\in Q_n^+}[\Ps(\RR^\Lam(\beta)],\\
[\Ps(\RR^\Lam)] & =\oplus_{\beta\in Q_n^+}[\Ps(\RR^\Lam(\beta)].
\end{aligned}$$

Let $\tilde{\rm K}_i$ be an endomorphism on $[\Ps(\RR^\Lam)]$ and $[\Rs(\RR^\Lam)]$ defined by \begin{equation}\label{2ki}
\tilde{\rm K}_i|_{[\Ps(\RR^\Lam(\beta)]}:=(q_i^2\pi_i)^{\<h_i,\Lam-\beta\>},\quad \tilde{\rm K}_i|_{[\Rs(\RR^\Lam(\beta)]}:=(q_i^2\pi_i)^{\<h_i,\Lam-\beta\>}.
\end{equation}
Note that $q^{-(\alpha_i|\alpha_j)}=q^{\rd_ia_{ij}}=q_i^{a_{ij}}$. Thus, applying \cite[(8.16)]{KKO2} we get that  \begin{equation}\label{effe}
{\rm E}_i{\rm F}_j-\pi^{\rmp(i)\rmp(j)}q_i^{-a_{ij}}{\rm F}_j{\rm E}_i=\delta_{ij}\frac{1-\tilde{\rm K}_i}{1-q_i^2\pi_i} ,
\end{equation}
which is the same as the fourth equality in Definition \ref{ug} if we identify ${\rm E}_i, {\rm E}_j$ with $e_i, f_j$ respectively.

Let $\Lam\in P^+$. Let $\mv(\Lam)$ be the $P$-weighted $\mathcal{U}_q(\mathfrak{g})$-module generated by $v_\Lam$ of weight $\Lam$ with additional relations given by:$$e_iv_\Lam=0,\qquad\qquad f_i^{\<h_i,\Lam\>+1}v_\Lam=0\quad\, {\text for\,\,all\,\,} i\in I.
$$ We define an $\mathbb{A}^\pi$-form of $\mv(\Lam)$ by $$
\mv_{\mathbb{A}^\pi}(\Lam):=\mcu_{\mathbb{A}^\pi}(\fg)v_\Lam .
$$

\begin{thm}\text{(\cite[Theorem 8.14]{KKO2})}\label{KKO2mainthm} Let $\Lam\in P^+$. Then $[\Ps(\RR^\Lam)]$ is a $\mcu_{\mathbb{A}^\pi}(\fg)$-module, and there is a $\mcu_{\mathbb{A}^\pi}(\fg)$-module isomorphism: $$
[\Ps(\RR^\Lam)]\cong \mv_{\mathbb{A}^\pi}(\Lam). $$
\end{thm}

By the definitions given above,
 $$
{\rm F}_i[\RR^\Lam(\beta)]=[\RR^\Lam(\beta+\alpha_i)e(\beta,i)],\quad {\rm E}_i[\RR^\Lam(\beta+\alpha_i)]=[e(\beta,i)\RR^\Lam(\beta+\alpha_i)].
$$
% \cite[Proposition 3.3]{OP}.

%
%
%Let $\Lam\in P^+$. Following \cite[(5.4)]{KKO2}, let $\mv(\Lam)$ be the $P$-weighted $\mcu(\fg)$-module generated by $v_\Lam$ of weight $\Lam$ with the defining relations given by: \begin{equation}\label{vlam}
%\wk_iv_\Lam=(q_i^2\pi_i)^{\<h_i,\Lam\>}v_\Lam, \quad e_iv_\Lam=0,\quad f_i^{\<h_i,\Lam\>+1}v_\Lam=0,\quad \text{for all $i\in I$.}
%\end{equation}

Recall that, for any $M=\oplus_{a\in\Z}(M_{a,\overline{0}}\oplus M_{a,\overline{1}})\in\rps(\RR(\beta))$, its $(q,\pi)$-dimension is given by: $$
\dim_q^{\pi}(M):=\sum_{a\in\Z}\bigl(\dim_{\bk}M_{a,\overline{0}}+\pi\dim_{\bk}M_{a,\overline{1}}\bigr)q^a\in\Z[q^{\pm1}]^\pi .
$$

\begin{lem}\label{keylem1}
Let $\Lam\in P^+$ and $\beta\in Q_n^+$. For any $\nu=(\nu_1,\cdots,\nu_n),\nu'=(\nu'_1,\cdots,\nu'_n)\in I^\beta$, we have $$e_{\nu_1}\cdots e_{\nu_n}f_{\nu'_n}\cdots f_{\nu'_1}v_\Lam=\dim_q^{\pi}\Bigl(e(\nu)\RR^\Lam(\beta)e(\nu')\Bigr)v_\Lam.
$$
\end{lem}

\begin{proof} By Theorem \ref{KKO2mainthm} and the definitions of the action of $E_i^\Lam$ and $F_i^\Lam$, we have $$\begin{aligned}
e_{\nu_1}\cdots e_{\nu_n}f_{\nu'_n}\cdots f_{\nu'_1}v_\Lam
&={\rm{E}}_{\nu_1}\cdots {\rm{E}}_{\nu_n}{\rm{F}}_{\nu'_n}\cdots {\rm{F}}_{\nu'_1}[\RR^\Lam(0)]
={\rm{E}}_{\nu_1}\cdots {\rm{E}}_{\nu_n}[\RR^\Lam(\beta)e(\nu')]\\
&=[e(\nu)\RR^\Lam(\beta)e(\nu')]=(\dim_q^\pi e(\nu)\RR^\Lam(\beta)e(\nu'))[\RR^\Lam(0)]\\
&=(\dim_q^\pi e(\nu)\RR^\Lam(\beta)e(\nu')) v_\Lam .
\end{aligned}
$$
This completes the proof of the lemma. \end{proof}

%\dim_q^\pi\Hom_{\RR^\Lam(\beta)}(\RR^\Lam(\beta)e(\nu),\RR^\Lam(\beta)e(\mu))=\dim_q^\pi\bigl(f_{\nu_n}\cdots f_{\nu_1}v_\Lam, f_{\nu'_n}\cdots f_{\nu'_1}v_\Lam\bigr).

\begin{lem}\label{action1} Let $\Lam\in P^+$ and $i,j_1,\cdots,j_k\in I$. Then
$$\wk_if_{j_1}\cdots f_{j_k}v_\Lam=\pi_i^{\<h_i,\Lam\>}q_i^{2\<h_i,\Lam-\alpha_{j_1}-\cdots-\alpha_{j_k}\>}f_{j_1}\cdots f_{j_k}v_\Lam$$
\end{lem}
\begin{proof} $$
\begin{aligned}
{\rm LHS}&=\bigl(\wk_if_{j_1}\wk_i^{-1}\bigr)\bigl(\wk_if_{j_2}\wk_i^{-1}\bigr)\cdots \bigl(\wk_if_{j_k}\wk_i^{-1}\bigr)\wk_iv_\Lam\\
&=q_i^{-2a_{ij_1}}\cdots q_i^{-2a_{ij_k}}(q_i^2\pi_i)^{\<h_i,\Lam\>}f_{j_1}\cdots f_{j_k}v_\Lam={\rm RHS}.
\end{aligned}$$
This completes the proof of the lemma.
\end{proof}

For each monomial of the form $f_{j_1}\cdots f_{j_n}$, we use the notation $f_{j_1}\cdots \widehat{f_{j_k}}\cdots f_{j_n}$ to denote the monomial obtained by removing $f_{j_k}$ from the monomial  $f_{j_1}\cdots f_{j_n}$. That is, $$
f_{j_1}\cdots \widehat{f_{j_k}}\cdots f_{j_n}:=f_{j_1}\cdots f_{j_{k-1}}f_{j_{k+1}}\cdots f_{j_n}.
$$
Similarly, for any $\beta\in Q_n^+$ and $\nu=(\nu_1,\cdots,\nu_n)\in I^\beta$, we define $$
(\nu_1,\cdots,\widehat{\nu_k},\cdots,\nu_n):=(\nu_1,\cdots,\nu_{k-1},\nu_{k+1},\cdots,\nu_n)\in I^{\beta-\alpha_{\nu_k}}.
$$

\begin{lem}\label{action2}  Let $\Lam\in P^+$ and $\beta\in Q_n^+$. For any $i\in I$ and $(j_1,\cdots,j_n)\in I^\beta$, we have
$$
\begin{aligned}e_if_{j_n}\cdots f_{j_1}v_\Lam &=\sum_{\substack{1\leq k\leq n\\ j_k=i}}f_{j_n}\cdots\widehat{f_{j_k}}\cdots f_{j_1}\frac{1-\pi_i^{\<h_i,\Lam\>}q_i^{2\<h_i,\Lam-\beta+\sum_{k\leq l\leq n}\alpha_{j_l}\>}}{1-q_i^2\pi_i}\\
&\qquad\qquad\times \pi_i^{\sum_{k<l\leq n}\rmp(j_l)}q_i^{-\<h_i,\sum_{k<l\leq n}\alpha_{j_l}\>}v_\Lam,
\end{aligned}
$$
where the summation is understood as zero if none of the $j_k$ is equal to $i$.
\end{lem}

\begin{proof} We use induction on $n$. If $n=1$, then the lemma holds by Definition \ref{ug} and (\ref{2ki}). Now suppose that $n>1$. We have that $$\begin{aligned}
{\rm LHS}&=(\pi_i^{\rmp(j_n)}q_i^{-a_{ij_n}}f_{j_n}e_i+\delta_{i,j_n}\frac{1-\tilde{K_i}}{1-q^2_i\pi_i})f_{j_{n-1}}\cdots f_{j_1}v_\Lam\\
&=\pi_i^{\rmp(j_n)}q_i^{-a_{ij_n}}f_{j_n}\Bigl(\sum_{\substack{1\leq k\leq n-1\\ j_k=i}}f_{j_{n-1}}\cdots\widehat{f_{j_k}}\cdots f_{j_1}\frac{1-\pi_i^{\<h_i,\Lam\>}q_i^{2\<h_i,\Lam-\beta+\alpha_{j_n}+\sum_{k\leq l\leq n-1}\alpha_{j_l}\>}}{1-q_i^2\pi_i}\\
&\qquad\qquad\times \pi_i^{\sum_{k<l\leq n-1}\rmp(j_l)}q_i^{-\<h_i,\sum_{k<l\leq n-1}\alpha_{j_l}\>}\Bigr)v_\Lam+\delta_{i,j_n}\frac{1-\tilde{K_i}}{1-q^2_i\pi_i}f_{j_{n-1}}\cdots f_{j_1}v_\Lam\\
&=\sum_{\substack{1\leq k\leq n-1\\ j_k=i}}f_{j_{n}}\cdots\widehat{f_{j_k}}\cdots f_{j_1}\frac{1-\pi_i^{\<h_i,\Lam\>}q_i^{2\<h_i,\Lam-\beta+\sum_{k\leq l\leq n}\alpha_{j_l}\>}}{1-q_i^2\pi_i}\pi_i^{\sum_{k<l\leq n}\rmp(j_l)}q_i^{-\<h_i,\sum_{k<l\leq n}\alpha_{j_l}\>}v_\Lam\\
&\qquad\qquad  +\delta_{i,j_n}\frac{1-\pi_i^{\<h_i,\Lam\>}q_i^{2\<\Lam-\alpha_{j_1}-\cdots-\alpha_{j_{n-1}}\>}}{1-q^2_i\pi_i}f_{j_{n-1}}\cdots f_{j_1}v_\Lam\\
&={\rm RHS}.
\end{aligned}
$$
where the second equality follows from induction hypothesis and we used Lemma \ref{action1} in the third equality.
\end{proof}

\begin{lem}\label{quanteq}
Let $\Lam\in P^+$ and $\beta\in Q_n^+$. For any $i\in I$, we have
$$
\frac{1-\pi_i^{\<h_i,\Lam\>}q_i^{2\<h_i,\Lam-\beta\>}}{1-q_i^2\pi_i}=[\<h_i,\Lam-\beta\>]^\pi_iq_i^{\<h_i,\Lam-\beta\>-1} .
$$
\end{lem}

\begin{proof}
We have $$\frac{1-\pi_i^{\<h_i,\Lam\>}q_i^{2\<h_i,\Lam-\beta\>}}{1-q_i^2\pi_i}=\frac{q_i^{-\<h_i,\Lam-\beta\>}-\pi_i^{\<h_i,\Lam\>}q_i^{\<h_i,\Lam-\beta\>}}{q_i^{-1}-q_i\pi_i}\frac{q_i^{\<h_i,\Lam-\beta\>}}{q_i}.
$$
Hence, to complete the proof, we only need to show for any simple root $\alpha_k$, $$
\pi_i^{\<h_i,\alpha_k\>}=1 .
$$
Actually, if $i\in I_{\rm{even}}$, then $\pi_i=1$; if $i\in I_{\rm{odd}}$, then by definition, $\<h_i,\alpha_k\>$ is even, hence $\pi_i^{\<h_i,\alpha_k\>}=1$.
\end{proof}

\begin{dfn}{\rm (\cite[Definition 3.2]{HS}\label{keydfn1})} For any $w\in\Sym_n$, $t\in\{1,2,\cdots,n\}$, we define $$
J_w^{<t}:=\{1\leq j<t|w(j)<w(t)\} .
$$
Let $\Lam\in P^+$. For any $\nu=(\nu_1,\cdots,\nu_n)\in I^n$ and $1\leq t\leq n$, we define  \begin{equation}\label{Ndef}
N^\Lam(w,\nu,t):=\<h_{\nu_t},\Lam-\sum_{j\in J_w^{<t}}\alpha_{\nu_j}\>.
\end{equation}
For any $\nu,\nu'\in I^n$, we define $\Sym(\nu,\nu'):=\big\{w\in\Sym_n | w\nu =\nu' \big\}$.
\end{dfn}

\begin{dfn} Let $w\in\Sym_n$. For any $\nu=(\nu_1,\cdots,\nu_n)\in I^n$ and $1\leq t\leq n$, we define \begin{equation}\label{pdef}
\rmp(w,\nu,t):=\sum_{\substack{1\leq k<t \\ w(k)>w(t)}}\rmp(\nu_k).
\end{equation}
\end{dfn}

\begin{lem}{\rm (\cite[Lemma 3.5]{HS})}\label{eqa1} Let $\Lam\in P^+$ and $\nu,\nu'\in I^n$. For any $w\in\Sym(\nu,\nu')$ and $1\leq t\leq n$, we have that $$N^\Lam(w,\nu,t)=\<h_{\nu_{t}},\Lam-\sum_{\substack{1\leq j<w(t),\\ j\in \{w (1),\cdots\, ,w (t-1)\}}}\alpha_{\nu'_{j}}\>.$$
\end{lem}

\begin{lem}\label{pequ} Let $\nu,\nu'\in I^n$. For any $w\in \Sym(\nu,\nu')$ and $1\leq t\leq n$, we have $$
\rmp(w,\nu,t)=\sum_{\substack{w(t)<k\leq n\\k\in \{w (1),\cdots\, ,w (t-1)\} }}\rmp(\nu'_k).
$$
\end{lem}

\begin{proof} For any $w(t)<k=w(i)\leq n$ with $i\in \{1,\cdots\, ,t-1\}$, we have $\nu'_k=\nu'_{w(i)}=\nu_i$ because $w\in\Sym(\nu,\nu')$.
The lemma follows at once from the definition of $\rmp(w,\nu,t)$.
\end{proof}

\medskip
\noindent
{\bf{Proof of Theorem \ref{mainthmA}}:} We claim that $$\begin{aligned}
&\quad\,\dim_q^{\pi}e(\nu)\RR^\Lam(\beta)e(\nu')\\
&=\sum_{\substack{(k_1,\cdots,k_n)\in\Sym_n(1,\cdots,n)\\ \nu_i=\nu'_{k_i}, \forall\,1\leq i\leq n}}\prod_{t=1}^{n}\Biggl(
\Bigl[\bigl(\Lam-\sum\limits_{\substack{1\leq i<k_t\\ i\neq k_s,\forall\,t\leq s\leq n}}\alpha_{\nu'_i}\bigr)(h_{\nu_t})\Bigr]^\pi_{\nu_t}q_{\nu_t}^{N^{\Lam}(1,\nu,t)-1}\pi_{\nu_t}^{\sum\limits_{\substack{i>k_t\\ i\neq k_s,\forall\,t\leq s\leq n}}\rmp(\nu'_{i})}\Biggr)
\end{aligned}
$$

We use induction on $|\beta|$. Suppose that the claim holds for any $\beta\in Q_{n-1}^+$. Now we assume $\beta\in Q_n^+$. Applying Lemmas \ref{keylem1}, \ref{action1}, \ref{action2} and \ref{quanteq}, we get that $$\begin{aligned}
&\quad\, \Bigl(\dim_{q}^\pi e(\nu)\RR^{\Lam}(\beta)e(\nu')\Bigl)v_{\Lam}\\
&=e_{\nu_1}\cdots\, e_{\nu_n}f_{\nu'_n} \cdots\, f_{\nu'_1}v_{\Lam}\\
&=\sum_{\substack{1\leq k_n\leq n\\ \nu_n=\nu'_{k_n}}}\frac{1-\pi_{\nu_n}^{\<h_{\nu_n},\Lam\>}q_{\nu_n}^{2\<h_{\nu_n},\Lam-\beta+\sum_{l\geq k_n}\alpha_{\nu'_l}\>}}{1-q_{\nu_n}^2\pi_{\nu_n}}\pi_{\nu_n}^{\sum_{l>k_n}\rmp(\nu'_l)}q_{\nu_n}^{-\<h_{\nu_n},\sum_{l>k_n}\alpha_{\nu'_l}\>}e_{\nu_{1}}\cdots\, e_{\nu_{n-1}}f_{\nu'_n} \cdots\, \widehat{f_{\nu'_{k_n}}} \\
&\qquad\qquad \times\cdots\times f_{\nu'_1}v_{\Lam}\\
&=\sum_{\substack{1\leq k_n\leq n\\ \nu_n=\nu'_{k_n}}} [\<h_{\nu_n},\Lam-\beta+\sum_{l\geq k_n}\alpha_{\nu'_l}\>]^\pi_{\nu_n}q_{\nu_n}^{\<h_{\nu_n},\Lam-\beta+\sum_{l\geq k_n}\alpha_{\nu'_l}\>-1}\pi_{\nu_n}^{\sum_{l>k_n}\rmp(\nu'_l)}q_{\nu_n}^{-\<h_{\nu_n},\sum_{l>k_n}\alpha_{\nu'_l}\>}\\
&\qquad\quad\times \dim_q^\pi e(\nu_1,\cdots,\nu_{n-1})\RR^\Lam(\beta-\alpha_{\nu_n})e(\nu'_1,\cdots,\widehat{\nu'_{k_n}},\cdots,\nu'_n)v_{\Lam}\\
&=\sum_{\substack{1\leq k_n\leq n\\ \nu_n=\nu'_{k_n}}}[\<h_{\nu_n},\Lam-\beta+\sum_{l\geq k_n}\alpha_{\nu'_l}\>]^\pi_{\nu_n}\pi_{\nu_n}^{\sum_{l>k_n}\rmp(\nu'_l)}{q_{\nu_{n}}^{1+(\Lambda-\beta)(h_{\nu_{n}})}}
\\
&\qquad\quad\times\Bigl(\dim_q^\pi e(\nu_1,\cdots,\nu_{n-1})\RR^{\Lam}(\beta-\alpha_{\nu_n})e(\nu'_1,\cdots,\widehat{\nu'_{k_n}} \cdots\, \nu'_n)\Bigl)v_{\Lam}.
\end{aligned} $$
It follows that  \begin{equation}\label{ind}\begin{aligned}
\dim_{q}^\pi\,e(\nu)\RR^{\Lam}(\beta)e(\nu')
&=\sum_{\substack{1\leq k_n\leq n\\ \nu_n=\nu'_{k_n}}}{q_{\nu_{n}}^{1+(\Lambda-\beta)(h_{\nu_{n}})}}\Bigl[(\Lam-\sum\limits_{i=1}^{k_{n}-1}\alpha_{\nu'_i})(h_{\nu_n})\Bigr]_{\nu_n}^\pi \pi_{\nu_n}^{\sum_{l>k_n}\rmp(\nu'_l)}\\
&\qquad\qquad\times\dim_q^\pi e(\nu_1,\cdots,\nu_{n-1})\RR^{\Lam}(\beta-\alpha_{\nu_n})e(\nu'_1,\cdots,\widehat{\nu'_{k_n}} \cdots\, \nu'_n).
\end{aligned}\end{equation}

We define $\tilde{\nu}'=(\tilde{\nu}'_1,\cdots,\tilde{\nu}'_{n-1}):=(\nu'_1,\cdots,\widehat{\nu'_{k_n}},\cdots,\nu'_n)$. Applying induction hypothesis, we can deduce that $$
\begin{aligned}
&\quad\, \Bigl(\dim_q^\pi e(\nu_1,\cdots\,\nu_{n-1})\RR^{\Lambda}(\beta-\alpha_{\nu_n})e(\nu'_1,\cdots,\widehat{\nu'_{k_n}},\cdots,\nu'_n)\Bigr)v_{\Lambda}\\
&=\Bigl(\dim_q^\pi e(\nu_1,\cdots\,\nu_{n-1})\RR^{\Lambda}(\beta-\alpha_{\nu_n})e(\tilde{\nu}'_1,\cdots,\tilde{\nu}'_{n-1})\Bigl)v_{\Lambda}\\
&=\sum_{\substack{(\tilde{k}_1,\cdots,\tilde{k}_{n-1})\in\Sym_{n-1}(1,\cdots,n-1)\\ \nu_i=\tilde{\nu}'_{\tilde{k}_i}, \forall\,1\leq i\leq n-1}}\prod_{t=1}^{n-1}
\Bigl(\Bigl[\bigl(\Lam-\sum\limits_{\substack{1\leq i<\tilde{k}_t\\ i\neq \tilde{k}_s,\forall\,t\leq s\leq n-1}}\alpha_{\tilde{\nu}'_i}\bigr)(h_{\nu_t})\Bigl]_{\nu_t}^\pi q_{\nu_t}^{N^{\Lam}(1,\nu,t)-1}\pi_{\nu_t}^{\sum\limits_{\substack{i>\tilde{k}_t\\ i\neq \tilde{k}_s,\forall\,t\leq s\leq n-1}}\rmp(\tilde{\nu}'_{i})}\Bigr)
v_{\Lambda}.
\end{aligned}
$$
For any given integer $1\leq k_n\leq n$, there is an associated natural bijection $\theta_{k_n}$ from the set $$
\Bigl\{(k_1,\cdots,k_{n-1})\Bigm|\begin{matrix}\text{$1\leq k_1,\cdots,k_{n-1}\leq n$, $\nu_i=\nu'_{k_i},\forall\,1\leq i\leq n-1$}\\ \text{$k_n\neq k_a\neq k_b, \forall\,1\leq a\neq b<n$}
\end{matrix}\Bigr\}
$$
onto the set $$
\Bigl\{(\tilde{k}_1,\cdots,\tilde{k}_{n-1})\Bigm|\begin{matrix}\text{$(\tilde{k}_1,\cdots,\tilde{k}_{n-1})\in\Sym_{n-1}(1,2,\cdots,n-1)$,}\\ \text{$\nu_i=\tilde{\nu}'_{\tilde{k}_i},\forall\,1\leq i\leq n-1$.}
\end{matrix}\Bigr\}
$$
which is defined by $$
\theta_{k_n}(k_1,\cdots,k_{n-1})=(\tilde{k}_1,\cdots,\tilde{k}_{n-1}),\quad
\tilde{k}_j:=\begin{cases}k_j, &\text{if $k_j<k_n$;}\\ k_j-1, &\text{if $k_j>k_n$.}\end{cases}\,\,\forall\,1\leq j\leq n-1.
$$
With this bijection $\theta_{k_n}$ in mind, we can deduce from the above calculation that $$
\begin{aligned}
&\quad\, \Bigl(\dim_q^\pi e(\nu_1,\cdots\,\nu_{n-1})\RR^{\Lambda}(\beta-\alpha_{\nu_n})e(\nu'_1,\cdots,\widehat{\nu'_{k_n}},\cdots,\nu'_n)\Bigl)v_{\Lambda}\\
&=\sum_{\substack{1\leq k_1,\cdots,k_{n-1}\leq n\\ \nu_i={\nu}'_{k_i}, \forall\,1\leq i\leq n-1\\ k_n\neq k_a\neq k_b,\forall\,1\leq a\neq b<n}}\prod_{t=1}^{n-1}
\Bigl(\Bigl[\bigl(\Lam-\sum\limits_{\substack{1\leq i<k_t\\ i\neq k_s,\forall\,t\leq s\leq n-1}}\alpha_{\nu'_i}\bigr)(h_{\nu_t})\Bigr]_{\nu_t}^\pi q_{\nu_t}^{N^{\Lam}(1,\nu,t)-1}\Bigr)\pi_{\nu_t}^{\sum\limits_{\substack{i>k_t\\ i\neq k_s,\forall\,t\leq s\leq n-1}}\rmp(\nu'_{i})}
v_{\Lambda}.
\end{aligned}
$$
Combining this with the equality (\ref{ind}), we prove our claim.

Finally, $\{k_1,\cdots,k_n\}$ is a permutation of $\{1,\cdots,n\}$ and $\nu_i={\nu}'_{k_i}, \forall\,1\leq i\leq n$ mean that there exists $w\in\Sym(\nu,\nu')$ such that $k_j=w(j)$, $\forall\,1\leq j\leq n$. Then it is clear that the theorem follows from our above claim and Lemma \ref{eqa1} and Lemma \ref{pequ}.
\qed\medskip

If we forget the $\Z_2$-grading or even the whole $(\Z\times\Z_2)$-grading, then we get the following corollary which recovers the results in \cite[Theorem 1.1, (1.2)]{HS}.

\begin{cor}\label{smaedimcor} Let $\Lam\in P^+$ and $\beta\in Q_n^+$. Let $\nu=(\nu_1,\cdots,\nu_n),\nu'=(\nu'_1,\cdots,\nu'_n)\in I^\beta$.
\begin{equation}\label{gradeddim}
\dim_{q}e(\nu)\RR^\Lam(\beta)e(\nu')=\sum_{\substack{w\in\Sym(\nu,\nu')}}\prod_{t=1}^{n}\Bigl([N^{\Lam}(w,\nu,t)]_{\nu_t}
q_{\nu_t}^{N^{\Lam}(1,\nu,t)-1}\Bigr);
\end{equation}

\begin{equation}\label{ungradeddim}
\dim e(\nu)\RR^{\Lambda}(\beta)e(\nu')=\sum\limits_{w\in\Sym(\nu,\nu')}\prod\limits_{t=1}^{n}N^\Lam(w,\nu,t).
\end{equation}
\end{cor}

\medskip
\noindent
{\bf{Proof of Corollary \ref{maincorA1}}:} It is clear that $e(\nu)\neq 0$ in $\R(\beta)$ if and only if $\dim e(\nu)\RR^{\Lambda}(\beta)e(\nu)\neq 0$. So Corollary \ref{maincorA1} follows from Corollary \ref{smaedimcor}.
\qed\medskip

\begin{examp}
We consider the type $B_2$ case, i.e. $$A=\begin{pmatrix}2\,&-2\\
-1\,&2\end{pmatrix}
$$ and let $1\in I_{\rm odd},\,2\in I_{\rm even}$. We choose $\nu=\nu'=(1,2,1),\,\Lam=2\Lam_1+\Lam_2$. Then one can easily check $$\Sym(\nu,\nu)=(1,(1,3)),$$ where $(1,3)$ denotes the transposition which swaps $1$ and $3$. In this case, $\rmp(1)=\overline{1}, \rmp(2)=\overline{0},\,d_1=1,\,d_2=2$, and $$\begin{aligned}
&N^\Lam(1,\nu,1)=N^\Lam(1,\nu,2)=N^\Lam(1,\nu,3)=2;\\
&N^\Lam((1,3),\nu,1)=N^\Lam((1,3),\nu,3)=2,\quad N^\Lam((1,3),\nu,2)=1;\\
&\rmp^\Lam(1,\nu,1)=\rmp^\Lam(1,\nu,2)=\rmp^\Lam(1,\nu,3)=\overline{0};\\
&\rmp^\Lam((1,3),\nu,1)=\overline{0},\quad \rmp^\Lam((1,3),\nu,2)=\rmp^\Lam((1,3),\nu,3)=\overline{1}.
\end{aligned}
$$ By Theorem \ref{mainthmA}, we have$$\begin{aligned}
\dim_q^\pi e(\nu)\RR^\Lam(2\alpha_1+\alpha_2) e(\nu)=&(\frac{(\pi q)^2-q^{-2}}{\pi q-q^{-1}}q)\cdot(\frac{q^4-q^{-4}}{q^2-q^{-2}}q^2)\cdot (\frac{(\pi q)^2-q^{-2}}{\pi q-q^{-1}}q)\\
&\quad\quad +(\frac{(\pi q)^2-q^{-2}}{\pi q-q^{-1}}q)\cdot q^2 \cdot (\frac{(\pi q)^2-q^{-2}}{\pi q-q^{-1}}q\pi)\\
=&q^8+3\pi q^6+4q^4+3\pi q^2+1.
\end{aligned}
$$
\end{examp}
%???????This Corollary (especially, \eqref{ungradeddim}) implies that almost results in \cite{HS} can be generalized to the super case. Explicitly, three criteria of $e(\nu)\neq 0$ (\cite[Theorem 1.3]{HS}, \cite[Theorem 3.24]{HS} and \cite[Corollary 3.35]{HS}) also works. The decomposation of ungraded dimension (\cite[Theorem 1.4]{HS} and \cite[Corollary 3.32]{HS}) and the indecomposability for special $\beta$ (\cite[Theorem 1.6]{HS}) still holds. The reason is these Theorems and Corollaries are deduced from \eqref{ungradeddim} and combinatorial methods but not involving the defining relations of cyclotomic quiver Hecke algebra. However, the construction of monomial bases (\cite[Theorem 1.5]{HS}) involves the defining relations and properties of cyclotomic quiver hecke algebra.

\bigskip

\section{Monomial bases of cyclotomic Odd nilHecke Algebra and of the bi-weight space $e(\wnu)\R(\beta)e(\wnu)$}
In this section, we will use the dimension formula obtained in Corollary \ref{smaedimcor} to construct explicitly certain bi-weight space $e(\wnu)\R(\beta)e(\wnu)$. In particular, this will yield
a monomial bases for the cyclotomic odd nilHecke algebra.

We first consider the cyclotomic odd nilHecke algebra $\ONH_n^\ell$.

\begin{dfn}{\rm (\cite[Section 5]{EKL})} Let $\ell,n\in\Z_{\geq 0}$. The odd nilHecke algebra $\ONH_n$ is the unital associative $K$-algebra with generators $\tau_1,\cdots,\tau_{n-1}, x_1,\cdots,x_{n-1}$ and the following defining relations:
$$\begin{aligned}
& \tau_i^2=0,\quad \tau_i\tau_{i+1}\tau_{i}=\tau_{i+1}\tau_{i}\tau_{i+1},\\
& x_i\tau_i+\tau_ix_{i+1}=1,\,\,\,\, \tau_i x_i+x_{i+1}\tau_i=1,\\
& x_ix_j+x_jx_i=0\,\, (i\neq j),\,\,\,\,\tau_i\tau_j+\tau_j\tau_i=0, (|i-j|>1),\\
& x_i\tau_j+\tau_jx_i=0\,\, (i\neq j,j+1).
\end{aligned}
$$
The cyclotomic odd nilHecke algebra $\ONH_n^\ell$ is defined to be the quotient of $\ONH_n$ by its two-sided ideal generated by $x_1^\ell$.
\end{dfn}
In particular, if we assume $0\in I_{\rm{odd}}$, then $\ONH_n=\RR(n\alpha_0)$, and $\ONH_n^\ell=\RR^{\ell\Lam_0}(n\alpha_0)$. We call $\ell>0$ the level of $\ONH_n^\ell$.

\begin{lem} Let $1\leq k<n$. Then for any $l\geq 1$, we have
\begin{equation}\label{keyformula}
\tau_k x_k^l=(-1)^l x_{k+1}^l \tau_k+\sum_{a+b=l-1}(-1)^bx_{k+1}^bx_k^a.
\end{equation}
\end{lem}

\begin{proof} We use induction on $l$. If $l=1$, then the lemma follows from the defining relation of $\ONH_n^\ell$. Suppose that the lemma holds for $l-1$. Now $$
\tau_k x_k^l=(\tau_k x_k)x_k^{l-1}=(1-x_{k+1}\tau_k)x_k^{l-1}=x_k^{l-1}-x_{k+1}(\tau_kx_k^{l-1}).
$$
Applying the induction hypothesis to $\tau_kx_k^{l-1}$, we prove the lemma.
\end{proof}

For each $w\in\Sym_n$, we define $\tau_w:=\tau_{i_1}\cdots\tau_{i_k}$, where $s_{i_1}\cdots s_{i_k}$ is a reduced expression of $w$. Then the braided relations of $\ONH_n^\ell$ ensures that $\tau_w$ is well-defined up to a sign.

\begin{lem}{\rm (\cite[(2.37)]{EKL})}\label{multiply} Let $u,v\in\Sym_n$. Then we have\[\tau_u\tau_v=
\begin{cases}
0,& \text{if $\ell(u)+\ell(v)\neq\ell(uv)$};\\
\pm\tau_{uv}, & \text{if $\ell(u)+\ell(v)=\ell(uv)$}.
\end{cases}\]
\end{lem}

Recall that for each $1\leq i<n$, $s_i:=(i,i+1)$ is the corresponding transposition.
\begin{dfn}
Let $\underline{w}=(s_{i_1},\cdots,s_{i_m})$ be an arbitrary expression of $w$ (i.e., $w=s_{i_1}\cdots s_{i_m}$). Let $\underline{e}=e_1\cdots e_m$ be a string of length $m$ such that $e_i\in\{0,1\}$ for each $i$. We define  \begin{equation}\label{le}
\ell(\underline{e}):=\sum_{i=1}^ke_i,\quad \underline{w}^{\underline{e}}:=s_{i_1}^{e_1}\cdots s_{i_m}^{e_m}.
\end{equation}
\end{dfn}

Recall that the ring of skew polynomials $\skp_a$ is defined to be $$
\skp_a:=\Z\<t_1,\cdots,t_a\>/\<t_it_j+t_jt_i=0,\,\,\text{for}\,\,i\neq j\> ,
$$
where $t_1,\cdots,t_a$ are skew commuting variables. We use the natural embedding $\skp_a\hookrightarrow \skp_{a+1}$ to identify $\skp_a$ as a subring of $\skp_{a+1}$. There are natural left (and right) actions of $\skp_n$ on $\ONH_n$ and $\ONH_n^\ell$ which are defined by multiplication followed with substituting each $t_j$ with $x_j$.

\begin{lem} Let $\ell,n\in\Z_{\geq 0}$. Then $\ONH_n^\ell\neq 0$ if and only if $\ell\geq n$.
\end{lem}

\begin{proof} By Corollary \ref{smaedimcor} and \cite[Corollary 3.22]{HS}, we see that the cyclotomic odd nilHecke algebra $\ONH_n^\ell$ has the same dimension as the usual cyclotomic nilHecke algebra $\HH_{\ell,n}^{(0)}$. So the lemma follows from \cite[(2.7)]{HuL}.
\end{proof}

Let $1\leq k<n$. We define $\underline{S_k}:=(s_k,s_{k-1},\cdots,s_1)$. For each $\underline{e}=(e_k,e_{k-1},\cdots,e_1)$ with $e_i\in\{0,1\},\forall\,1\leq i\leq k$, we define $$
\tau_{\underline{S_k}^{\underline{e}}}=\tau_{s_k^{e_k}}\tau_{s_{k-1}^{e_{k-1}}}\cdots \tau_{s_1^{e_1}},
$$
where we understand $\tau_{1}$ as $1$. The following Lemma is crucial in constructing the monomial basis of $\ONH_n^\ell$.

\begin{lem}\label{keylem2} Let $\ell,n\in\Z_{\geq 0}$ with $\ell\geq n$. For each $1\leq k<n$, there is a set of skew polynomials $\{g_{\underline{e},k+1}\in\skp_{k+1}|\underline{e}=(e_k,\cdots,e_1), e_i\in\{0,1\},\forall\,1\leq i\leq k\}$ such that \begin{enumerate}
\item each $g_{\underline{e},k+1}$ is a polynomial in $t_{k+1}$ of degree $\ell-k+\ell(\underline{e})$ with leading coefficient invertible and other coefficients in $\skp_k$.
\item $\sum_{\underline{e}}g_{\underline{e},k+1}\tau_{\underline{S_k}^{\underline{e}}}=0$ holds in $\ONH_n^\ell$.
\end{enumerate}
\end{lem}

\begin{proof} We use induction on $k$. If $k=1$, then \eqref{keyformula} implies $$0=\tau_1 x_1^\ell=(-1)^\ell x_2^\ell \tau_1+\sum_{a+b=\ell-1}(-1)^bx_2^bx_1^a.
$$ Hence the lemma follows in this case because we can take $$
g_{0,2}=\sum_{a+b=\ell-1}(-1)^bx_2^bx_1^a,\quad \,g_{1,2}=(-1)^\ell x_2^\ell. $$

Suppose the lemma holds for an integer $1<k<n-1$. We want to prove it holds for the integer $k+1$. By induction hypothesis,  there is a set of skew polynomials $\{g_{\underline{e},k+1}\in\skp_{k+1}|\underline{e}=(e_1,\cdots,e_k), e_i\in\{0,1\},\forall\,1\leq i\leq k\}$ such that $$ 0=\sum_{\underline{e}}g_{\underline{e},k+1}\tau_{\underline{S_k}^{\underline{e}}},
$$
where  $g_{\underline{e},k+1}$ is a polynomial in $t_{k+1}$ of degree $\ell-k+\ell(\underline{e})$ with leading coefficient invertible and other coefficients in $\text{SkPol}_k$. In particular, \begin{equation}\label{induction} 0=\tau_{k+1}\sum_{\underline{e}}g_{\underline{e},k+1}\tau_{\underline{S_k}^{\underline{e}}}.
\end{equation}
Applying \eqref{keyformula} we can write that $$
\tau_{k+1}g_{\underline{e},k+1}=h_{k+2}+h'_{k+2}\tau_{k+1},
$$
where $h_{k+2}$ (resp., $h'_{k+2}$) is a polynomial in $t_{k+2}$ of degree $\ell-k-1+\ell(0\underline{e})$ (resp., $\ell-k+\ell(\underline{e})=\ell-k-1+\ell(1\underline{e})$) with leading coefficient invertible and other coefficients in $\skp_{k+1}$.
Now we define $g_{0\underline{e},k+2}:=h_{k+2},\, g_{1\underline{e},k+2}:=h'_{k+2}$. Then the lemma follows from \eqref{induction} and Lemma \ref{multiply}.
\end{proof}

\begin{thm}\label{mainthm2} Let $\ell,n\in\Z_{\geq 0}$ with $\ell\geq n$. Then the following set \begin{equation}\label{bases1}
\Sigma:=\{x_{1}^{k_1}\cdots x_{n}^{k_n}\tau_w\,|\,w\in\Sym_n,\,\,0\leq k_i\leq\ell-i,\forall\,1\leq i\leq n\}
\end{equation} forms a $\bk$-basis of $\ONH_n^\ell$.
\end{thm}

\begin{proof} In fact, by  Corollary \ref{smaedimcor} and \cite[Corollaries 3.8, 3.22]{HS}, we can deduce that \begin{equation}\label{odddim}
\dim\ONH_n^\ell=n!\prod_{i=0}^{n-1}(\ell-i).
\end{equation} Hence we only need to show that the set $\Sigma$ is a set of $\bk$-linear generators of $\ONH_n^\ell$. For each integer $1\leq m\leq n$, we set $$\widetilde{\ONH}_{\leq m}^\ell
:=\text{$\bk$-span}\{x_{1}^{k_1}\cdots x_{m}^{k_m}\tau_w\,|\,w\in\Sym_n\,\,0\leq k_i\leq\ell-i,\forall\,1\leq i\leq m\},
$$
We claim that \begin{equation}\label{generating}
x_{1}^{t_1}\cdots x_{m}^{t_m}\tau_w\in \widetilde{\ONH}_{\leq m}^\ell,\quad\forall\,w\in\Sym_n,\, t_1,\cdots,t_m\in\Z_{\geq 0} .
\end{equation}
We prove \eqref{generating} by induction upward on $m$, downward on $\ell(w)$ and then upward on $t_m$. If $m=1,\,\ell(w)=n(n-1)/2,\,t_m=0$, then clearly \eqref{generating} holds. Now we consider the term $x_{1}^{t_1}\cdots x_{m}^{t_m}\tau_w$ in general case. If $t_m\leq\ell-m$, then by induction hypothesis, $$
x_{1}^{t_1}\cdots x_{m-1}^{t_{m-1}}\tau_w\in \widetilde{\ONH}_{\leq m-1}^\ell,
$$ and hence $$
x_{1}^{t_1}\cdots x_{m}^{t_{m}}\tau_w\in \widetilde{\ONH}_{\leq m}^\ell.
$$

If $t_m>\ell-m$, then applying Lemma \ref{keylem2} we can deduce that \begin{equation}\label{transformation}
x_m^{\ell-(m-1)}=h_m\cdot 1+\sum_{w'\neq 1}f_{m,\,w'}\tau_{w'},
\end{equation}
where $h_m\in\skp_m$ is a polynomial of $t_m$ of degree $\leq\ell-m$ with coefficients in $\text{SkPol}_{m-1}$, $f_{m,\,w'}\in\text{SkPol}_m$. Now, left-multiplying \eqref{transformation} by $x_{1}^{t_1}\cdots x_{m}^{t_{m}-(\ell-m+1)}$ and post-multiplying $\tau_w$ and using the Lemma \ref{multiply}, we can get that $$ x_{1}^{t_1}\cdots x_{m}^{t_{m}}\tau_w=h'_m\tau_w+\sum_{w'\neq 1}f'_{m,\,w'}\tau_{w'},
$$  where $h'_m$ is a polynomial of $x_m$ of degree $\leq t_m-1$ with coefficients in $\text{SkPol}_{m-1}$, $f_{m,\,w'}\in\text{SkPol}_m$, and $f_{m,\,w'}\neq 0$ only if $\ell(w')>\ell(w)$ (by Lemma \ref{multiply}). Applying induction hypothesis on the right-hand side of the equation above we complete the proof of \eqref{generating}. Finally, \eqref{generating} implies that $\widetilde{\ONH}_{\leq n}^\ell=\ONH_n^\ell$. Hence the statement comes from \eqref{odddim}.
\end{proof}

In the rest of this section, we shall generalise Theorem \ref{mainthm2} to a more general bi-weight subspace of an arbitrary cyclotomic quiver Hecke superalgebra as we did for the usual cyclotomic quiver Hecke algebras in \cite[Theorem 4.8]{HS}. To do this, we need some notations. We fix $p\in\N$, $\fb:=(b_1,\cdots,b_p)\in\N^p$ and $\nu^1,\cdots,\nu^p\in I$ such that $\nu^i\neq\nu^j$ for any $1\leq i\neq j\leq p$ and $\sum_{i=1}^{p}b_i=n$.
Recall that in (\ref{wnu0}) we have defined $$
\wnu=(\wnu_1,\cdots,\wnu_n):=\bigl(\underbrace{\nu^1,\cdots,\nu^1}_{\text{$b_1$ copies}},\cdots,\underbrace{\nu^p,\cdots,\nu^p}_{\text{$b_p$ copies}}\bigr)\in I^\beta ,
$$
where $\beta=\sum_{i=1}^{p}b_i\alpha_{\nu^i}\in Q_n^+$. Let $\Lam\in P^+$ be an arbitrary integral dominant weight. Our purpose is to give a monomial basis for the bi-weight subspace $e(\wnu)\RR^\Lam(\beta)e(\wnu)$, which will reduce to the monomial basis we constructed in Theorem \ref{mainthm2} in the special case when $p=1$, $\Lam=\ell\Lam_{\nu^1}$ and $\nu^1\in I_{\rm{odd}}$.

We set $b_0:=0, c_t:=\sum_{i=0}^{t}b_i$ for any $0\leq t\leq p$. We set \begin{equation}\label{sb}
\Sym_{\fb}:=\Sym_{\{1,\cdots,c_1\}}\times\Sym_{\{c_1+1,\cdots,c_2\}}\times\cdots\times\Sym_{\{c_{p-1}+1,\cdots,n\}}.
\end{equation}
For each $1\leq t\leq p$, we define $$
N^\Lam_t(\wnu):=N^\Lam(1,\wnu,c_{t-1}+1).
$$ When $c_{i-1}<k\leq c_{i}$ for some $1\leq i\leq p$, we define \begin{equation}\label{nwvk}
N^\Lam(\wnu,k)=N^\Lam_{i}(\wnu)-(k-c_{i-1}-1).
\end{equation}
%Now we want to construct monomial bases for the subspace $e(\wnu)\RR^\Lam(\beta)e(\wnu)$.

The following corollary is an analogue of \cite[Theorem 4.4]{HS} for cyclotomic quiver Hecke superalgebras.

\begin{cor}\label{symdim}  Let $\Lam\in P^+$. Let $\wnu$ be given as in (\ref{wnu0}). Let $\beta\in Q_n^+$ such that $\wnu\in I^\beta$. Then we have
$$ \dim e(\wnu)\RR^{\Lam}(\beta)e(\wnu)=\prod_{i=1}^{p}\Bigl(b_{i}!\prod_{j=0}^{b_i-1}(N^\Lam_{i}(\wnu)-j)\Bigr)=
\Bigl(\prod_{i=1}^{p}b_{i}!\Bigr)\Bigl(\prod_{k=1}^{n}N^\Lam(\wnu,k)\Bigr). $$
In particular, $e(\wnu)\neq 0$ in $\RR^\Lam(\beta)$ if and only if $N^\Lam_{i}(\wnu)\geq b_{i}$ for any $1\leq i\leq p$.
\end{cor}

\begin{proof} This follows directly from \eqref{ungradeddim}.
\end{proof}

Recall $$
\mathcal{P}_\nu:=\bk\<{\rm x}_1,\cdots,{\rm x}_n\>/\<{\rm x}_a{\rm x}_b-(-1)^{\rmp(\nu_a)\rmp(\nu_b)}{\rm x}_b{\rm x}_a|1\leq a<b\leq n\> .
$$
For any $1\leq m\leq n$, we set $\nu_{\leq m}:=(\nu_1,\cdots,\nu_{m})$. There is a natural embedding $\mathcal{P}_{\nu_{\leq m}} \hookrightarrow \mathcal{P}_\nu$ which identifies $\mathcal{P}_{\nu_{\leq m}}$ as a subring of $\mathcal{P}_\nu$.
%Again, there are natural left (and right) actions of $\mathcal{P}_\nu$ on $e(\nu)\RR(\beta)e(\nu)$ and $e(\nu)\RR^\Lam(\beta)e(\nu)$.

\begin{lem}\label{annihilator poly} Let $\Lam\in P^+$. Let $\wnu$ be given as in (\ref{wnu0}). Let $\beta\in Q_n^+$ such that $\wnu\in I^\beta$. Then for each $1\leq i\leq p$, there exists an element $q_i\in \mathcal{P}_{\wnu_{\leq  c_{i-1}+1}}$ which is a polynomial in ${\rm x}_{c_{i-1}+1}$ of degree $N^\Lam_{i}(\wnu)$ with leading coefficient invertible and other coefficients in $\mathcal{P}_{\wnu_{\leq c_{i-1}}}$, and such that
$q_ie(\wnu)=0$ in $e(\wnu)\RR^\Lam({\beta})e(\wnu)$.
\end{lem}

\begin{proof} Let $\widehat{\nu}$ be the $n$-tuple obtained by moving the $(c_{i-1}+1)$-th component of $\wnu$ (which is exactly $\nu^{i}$) to the first position and unchanging the relative positions of all the other components.
Consider $$\tau_{c_{i-1}}\tau_{c_{i-1}-1}\cdots\tau_{1}a^{\Lambda}_{\nu^{i}}(x_{1})e(\widehat{\nu})\tau_{1}\tau_{2}\cdots\tau_{c_{i-1}}.$$  By definition $a^{\Lambda}_{\nu^{i}}(x_{1})e(\widehat{\nu})=0$ in $\RR^\Lam(\beta)$. On the other hand, since $\nu^i\neq\nu^t$ for any $1\leq t<i$, we can deduce from the commutator relations in Definition \ref{qhs} that  $$
\tau_{c_{i-1}}\tau_{c_{i-1}-1}\cdots\tau_{1}a^{\Lambda}_{\nu^{i}}(x_{1})e(\widehat{\nu})=
(\pm a^{\Lambda}_{\nu^{i}}(x_{c_{i-1}+1}))\tau_{c_{i-1}}\tau_{c_{i-1}-1}\cdots\tau_{1}e(\widehat{\nu}).
$$
Also, from the defining relations in Definition \ref{qhs} we can get that $$
\tau_{c_{i-1}}\tau_{c_{i-1}-1}\cdots\tau_{1}e(\widehat{\nu})\tau_{1}\tau_{2}\cdots\tau_{c_{i-1}}
=\prod_{\substack{t=1}}^{i-1}\prod_{d=c_{t-1}+1}^{c_{t}}\widetilde{Q}_{\nu^{t},\nu^i}({\rm x}_d,{\rm x}_{c_{i-1}+1})e(\wnu),
$$
where $\widetilde{Q}_{\nu^{t},\nu^i}({\rm x}_d,{\rm x}_{c_{i-1}+1})$ is a polynomial on ${\rm x}_d,{\rm x}_{c_{i-1}+1}$ whose monomials are in one-to-one correspondence with the monomials of ${Q}_{\nu^{t},\nu^i}({\rm x}_d,{\rm x}_{c_{i-1}+1})$ and each monomial of $\widetilde{Q}_{\nu^{t},\nu^i}({\rm x}_d,{\rm x}_{c_{i-1}+1})$ differs with the corresponding monomials of ${Q}_{\nu^{t},\nu^i}({\rm x}_d,{\rm x}_{c_{i-1}+1})$ by $
\pm 1$. Now we define $$q_i=\pm {a^{\Lambda}_{\nu^{i}}}({\rm x}_{c_{i-1}+1})\prod_{\substack{t=1}}^{i-1}\prod_{d=c_{t-1}+1}^{c_{t}}\widetilde{Q}_{\nu^{t},\nu^{i}}({\rm x}_{d},{\rm x}_{c_{i-1}+1}).$$
A direct computation shows this is a polynomial in ${\rm x}_{c_{i-1}+1}$ of degree $N^\Lam_{i}(\wnu)$ with leading coefficient invertible and other coefficients in $P_{\wnu_{\leq c_{i-1}}}$, and it satisfies that
$q_ie(\wnu)=0$ in $e(\wnu)\RR^\Lam({\beta})e(\wnu)$.
\end{proof}

\begin{lem}\label{keylem3} Let $\Lam\in P^+$. Let $\wnu$ be given as in (\ref{wnu0}). Let $\beta\in Q_n^+$ such that $\wnu\in I^\beta$. Suppose that $e(\wnu)\neq 0$ in $\RR^\Lam(\beta)$.
For each $1\leq i\leq p$ and $c_{i-1}<k< c_{i}$, we define $\underline{\widetilde{S}_k}:=(s_k,s_{k-1},\cdots,s_{c_{i-1}+1})$. Then there is a set $$
\Sigma:=\bigl\{q_{\underline{e},k+1}\in P_{\wnu_{\leq k+1}}\bigm|\underline{e}=(e_k,e_{k-1},\cdots,e_{c_{i-1}+1}),\,
e_i\in\{0,1\},\forall\,c_{i-1}+1\leq i\leq k\bigr\}$$ such that \begin{enumerate}
\item each $q_{\underline{e},k+1}$ is a skew polynomial in ${\rm x}_{k+1}$ of degree $N^\Lam_{i}(\wnu)-(k-c_{i-1}-\ell(\underline{e}))$ with leading coefficient invertible and other coefficients in $P_{\wnu_{\leq k}}$.
\item $\sum_{\underline{e}}q_{\underline{e},k+1}\tau_{\underline{\widetilde{S}_k}^{\underline{e}}}=0$ holds in $e(\wnu)\RR^\Lam({\beta})e(\wnu)$.
\end{enumerate}
\end{lem}

\begin{proof} By Corollary \ref{symdim}, the assumption that $e(\wnu)\neq 0$ implies that $N^\Lam_{i}(\wnu)\geq b_{i}$ for any $1\leq i\leq p$. Let $1\leq i\leq p$ and $c_{i-1}< k< c_{i}$.

If $k=c_{i-1}+1$, then the lemma follows directly from Lemma \ref{annihilator poly}. In general, we use induction on $k-c_{i-1}$. Suppose the lemma holds when  $k-c_{i-1}<m$. We now assume that $k-c_{i-1}=m$. Set $\underline{\widetilde{S}_{k-1}}:=(s_{k-1},\cdots,s_{c_{i-1}})$.
By induction hypothesis,  there is a set of skew polynomials $\{q_{\underline{e}',k}\in\skp_{k}|\underline{e}'=(e_1,\cdots,e_{k-1}), e_i\in\{0,1\},\forall\,1\leq i\leq k-1\}$ such that $$ 0=\sum_{\underline{e}'}q_{\underline{e}',k}\tau_{\underline{\tilde{S}_{k-1}}^{\underline{e}'}},
$$
where $q_{\underline{e}',k}$ is a skew polynomial in ${\rm x}_{k}$ of degree $N^\Lam_{i}(\wnu)-(k-1-c_{i-1}-\ell(\underline{e}'))$ with leading coefficient invertible and other coefficients in $\text{SkPol}_{k-1}$. In particular, \begin{equation}\label{induction2} 0=\tau_{k}\sum_{\underline{e}}q_{\underline{e}',k}\tau_{\underline{\tilde{S}_{k-1}}^{\underline{e}'}}.
\end{equation}
Applying the commutator relations between $\tau_k$ and $x_k,x_{k+1}$ we can deduce that $$
\tau_{k}q_{\underline{e}',k}=h_{k+1}+h'_{k+1}\tau_{k},
$$
where $h_{k+1}$ (resp., $h'_{k+1}$) is a polynomial in ${\rm x}_{k+1}$ of degree $N^\Lam_{i}(\wnu)-(k-c_{i-1}-\ell(\underline{e}'))$ (resp., $N^\Lam_{i}(\wnu)-(k-1-c_{i-1}-\ell(\underline{e}'))=N^\Lam_{i}(\wnu)-(k-c_{i-1}-\ell(1\underline{e}'))$) with leading coefficient invertible and other coefficients in $\skp_{k}$.
Now we define $g_{0\underline{e}',k+1}:=h_{k+1},\, g_{1\underline{e}',k+1}:=h'_{k+1}$. Then the lemma follows from \eqref{induction2} and Lemma \ref{multiply}.
\end{proof}

%
%
%
%
%
%
%
%
%
%In this case, we have $$
%\tau_k x_k^le(\wnu)=(-1)^{l\rmp(\nu^i)} x_{k+1}^l \tau_ke(\wnu)+(-1)^{\rmp(\nu^i)+1}\sum_{a+b=l-1}(-1)^{b\rmp(\nu^i)}x_{k+1}^bx_k^ae(\wnu).$$
%
%This follows from Lemma \ref{annihilator poly} and an induction on $k$. Note that when
%
%
%The proof is exactly the same as Lemma \ref{keylem2}.

\begin{thm}\label{mainthm3} Let $\Lam\in P^+$. Let $\wnu$ be given as in (\ref{wnu0}). Let $\beta\in Q_n^+$ such that $\wnu\in I^\beta$.
The following set \begin{equation}\label{base1}
\Bigl\{e(\wnu)\prod_{k=1}^{n}x_{k}^{r_{k}}\tau_{w }\Bigm|\begin{matrix}\text{$w\in\Sym_{\fb}$, for any $1\leq i\leq p$, $c_{i-1}<k\leq c_{i}$,}\\
\text{$0\leq r_{k}<N^\Lam(\wnu,k)$}\end{matrix}\Bigr\}
\end{equation} forms a $\bk$-basis of $e(\wnu)\RR^\Lam({\beta})e(\wnu)$.
\end{thm}

\begin{proof} We follow a similar idea used in the proof of Theorem \ref{mainthm2}. For each $1\leq m\leq n$, we define $$\Sigma_{\leq m}:=\{x_{1}^{k_1}\cdots x_{m}^{k_m}\tau_w\,|\,w\in\Sym_{\fb},\,0\leq k_j<N^\Lam(\wnu,j),\forall\,1\leq j\leq m\}.
$$
By Theorem \ref{symdim}, we know that $$
\dim e(\wnu)\RR^\Lam({\beta})e(\wnu)=\# \Sigma_{\leq n}.
$$
Therefore, to prove the theorem, it suffices to show that $\Sigma_{\leq n}$ spans the whole $\bk$-space $e(\wnu)\RR^\Lam({\beta})e(\wnu)$. We claim that for any $1\leq m\leq n$, any $w\in\Sym_{\fb},\,t_1,\cdots,t_m\geq 0$, \begin{equation}\label{generating2}
x_{1}^{t_1}\cdots x_{m}^{t_m}\tau_w\in\text{$\bk$-Span} \{z|z\in\Sigma_{\leq m}\} .
\end{equation}

In fact, using Lemma \ref{annihilator poly} and Lemma \ref{keylem3}, we can prove the above claim by induction upward on $m$, downward on $\ell(w)$ then upward on $t_m$ which is exactly the same argument used in the proof of Theorem \ref{mainthm2}. In particular, $\Sigma_{\leq n}$ is a $\bk$-linear spanning set of $e(\wnu)\RR^\Lam({\beta})e(\wnu)$. This completes the proof of the theorem.
\end{proof}

\bigskip

\section{Monomial bases and Indecomposability of some $\RR^\Lam(\beta)$}
Throughout this section, we shall assume that \begin{equation}\label{assump1}
\beta=\alpha_{1}+\alpha_{2}+\cdots+\alpha_{n}, \quad \alpha_i\neq\alpha_j,\,\forall\,1\leq i\neq j\leq n
\end{equation}
The purpose of this section is to construct an explicit monomial basis for $\RR^{\Lambda}(\beta)$ and show that $\RR^{\Lambda}(\beta)$ is indecomposable under the assumption (\ref{assump1}).

\begin{lem}\label{Lm41} Assume (\ref{assump1}) holds. Let $\mu\in I^{\beta}$. Then $e(\mu)\neq 0$ if and only if $N^\Lam(1,\mu,k)>0$ for any $1\leq k\leq n$. In this case, $\dim e(\mu)\RR^\Lam(\beta)e(\mu)=\prod_{k=1}^{n}N^\Lam(1,\mu,k)$.
\end{lem}

\begin{proof} This follows directly from Theorem \ref{symdim}.
\end{proof}

Note that the assumption (\ref{assump1}) implies that for $\mu,\nu\in I^\beta$, there exists a unique $d_{\mu,\nu}\in\Sym_n$ such that $d_{\mu,\nu}\mu=\nu$. Using Corollary \ref{smaedimcor} we see that $\dim e(\mu)\RR^\Lam(\beta)e(\nu)=\prod_{k=1}^{n}N^\Lam(d_{\mu,\nu},\mu,k)$. Thus, we have the following result.

\begin{cor}\label{cor42}  Assume (\ref{assump1}) holds. Let $\mu,\nu\in I^{\beta}$. The $e(\mu)\RR^\Lam(\beta)e(\nu)\neq 0$
if and only if $N^\Lam(d_{\mu,\nu},\mu,k)\neq 0$ for any $1\leq k\leq n$.
\end{cor}

\begin{lem}\label{keylem4} Assume (\ref{assump1}) holds. Let $1\leq k\leq n$ with $N^\Lam(d_{\mu,\nu},\mu,k)>0$. Then there exists an element $p_{k}\in \mathcal{P}_{\nu_{\leq k}}$ which can be viewed as a polynomial in ${\rm x}_{k}$ of degree $N^\Lam(d_{\mu,\nu},\mu,k)$ with leading coefficient invertible and other coefficients in $\mathcal{P}_{\nu_{< k}}$. Moreover, $ \tau_{d_{\mu,\nu}}p_{k}e(\mu)$ is a zero element in $e(\nu)\RR^\Lam({\beta})e(\mu)$.
\end{lem}

\begin{proof}
Suppose $\mu_k=\nu_i$, where $1\leq i\leq n$. We define $\mathcal{J}_i:=\{1\leq m<k|d_{\mu,\nu}(m)>{i}\}$ and write $$
\mathcal{J}_i=\{m_j|1\leq j\leq g, 1\leq m_{1}<m_{2}<\cdots<m_{g}<k\}.
$$
Then $\mathcal{J}_i=\{1\leq m\leq k|\mu_m=\nu_t, i<t\leq n\}$.

We consider the following products of cycles: $$\begin{aligned}
u_1:&=(k-g+1,k-g,\cdots,m_1+1,m_1)(k-g+2,k-g+1,\cdots,m_2+1,m_2)\cdots \\
&\qquad\qquad (k,k-1,\cdots,m_g+1,m_g) .\end{aligned}
$$
We have $$
u_1=(s_{k-g}\cdots s_{m_{1}+1}s_{m_{1}})(s_{k-g+1}\cdots s_{m_2+1}s_{m_2})\cdots (s_{k-1}\cdots s_{m_g+1}s_{m_g}),
$$
and this is a reduced expression of $u_1$. We set $\mu^{[1]}:= u_1\mu$. In other words, $\mu^{[1]}$ is obtained from $\mu$ by moving its $m_1$-th, $\cdots$, $m_g$-th components to the $(k-g+1)$-th, $\cdots$, $k$-th positions respectively, and unchanging the relative positions of all the remaining components of $\mu$.
In particular, we have $\mu^{[1]}_{k-g}=\mu_k=\nu_i$.

Let $\widehat{\mu}$ be the $n$-tuple obtained from $\mu^{[1]}$ by moving the $(k-g)$-th component $\mu^{[1]}_{k-g}$ (which is equal to $\mu_k$ by construction) of $\mu^{[1]}$  to the first position and unchanging the relative positions of all the other components. We consider $$\tau_{k-g}\tau_{k-g-1}\cdots\tau_2\tau_{1}a^{\Lambda}_{\nu^{i}}(x_{1})e(\widehat{\mu})\tau_{1}\tau_{2}\cdots\tau_{k-g-1}\tau_{k-g}.
$$
The same argument as in the proof of Lemma \ref{annihilator poly} shows that this equals to $p_k^{[1]}e(\mu^{[1]})$, where $p_k^{[1]}$ can be viewed as a polynomial in ${\rm x}_{k-g}$ of degree $N^\Lam(d_{\mu,\nu},\mu,k)$ with leading coefficient in $\bk^\times$ and other coefficients in $\mathcal{P}_{\mu^{[1]}_{< k-g}}$. Clearly, this is zero in $\RR^\Lam({\beta})e(\mu^{[1]})$.

By the commutator relations between $\tau_r$ and $x_s$, we can find a skew polynomial $p_{k}\in\mathcal{P}_{\nu_{\leq k}}$ in ${\rm x}_{k}$ of degree $N^\Lam(d_{\mu,\nu},\mu,k)$ with leading coefficient invertible and other coefficients in $\mathcal{P}_{\nu_{< k}}$ and such that $p_k^{[1]}\tau_{u_1}e(\mu)=\tau_{u_1}p_k e(\mu)$. In fact, up to a sign on each monomial, $p_{k}$ is equal to $u_1^{-1}(p_k^{[1]})$. Thus we have $$\begin{aligned}\label{poly}
\tau_{u_1}p_k e(\mu)=p_k^{[1]}\tau_{u_1}e(\mu)=p_k^{[1]}e(\mu^{[1]})\tau_{u_1}=0.
\end{aligned}
$$
Finally, by construction we can find $u_2\in\Sym_n$ such that $d_{\mu,\nu}=u_2u_1$, and $\ell(d_{\mu,\nu})=\ell(u_2)+\ell(u_1)$. Note that under our assumption (\ref{assump1}), $\tau_{d_{\mu,\nu}}$ depends only on $d_{\mu,\nu}$ but not on the choices of the reduced expression of $d_{\mu,\nu}$. Hence we complete our proof.
% (This is from \eqref{poly} ).
\end{proof}

\medskip
\noindent
{\bf{Proof of Theorem \ref{mainthmC} Part 1)}:} This follows from Lemma \ref{keylem4} and Corollary \ref{smaedimcor}.\qed\medskip
%\begin{thm} Assume (\ref{assump1}) holds. Let $\mu,\nu\in I^\beta$ such that $e(\nu)\RR^\Lam(\beta)e(\mu)\neq 0$. Then the elements in the following set $$
% \Bigl\{\tau_{d_{\mu,\nu}}\prod_{k=1}^{n}x_{k}^{r_{k}}e(\mu)\Bigm|\, 0\leq r_{k}<N^\Lam(d_{\mu,\nu},\mu,k)\Bigr\}.
%$$
%form a $K$-linear basis of $e(\nu)R^{\Lambda}(\beta)e(\mu)$.
%\end{thm}
%
%\begin{proof}
%
%\end{proof}

For each $\beta\in Q_n^+$, let $\mathcal{Z}:=Z(\RR^\Lam(\beta))$ be the center of $\RR^\Lam(\beta)$. Then $\mathcal{Z}$ is naturally $\Z$-graded. Let $\mathcal{Z}_0$ be the degree $0$ component of $\mathcal{Z}$. It is well-known that the quiver Hecke superalgebra $\RR(\beta)$ is indecomposable because $\mathcal{Z}_0=\bk e(\beta)$ by \cite[\S4.5]{HW},  where $e(\beta):=\oplus_{\nu\in I^\beta}e(\nu)$. However, this is unclear for the cyclotomic quiver Hecke superalgebra $\RR^\Lam(\beta)$. Even for the usual cyclotomic quiver Hecke algebra, this is unclear except for some special cases. In the rest of this section, we shall study the indecomposability of the quiver Hecke superalgebra $\RR^\Lam(\beta)$.

\begin{lem}\label{lem43} Assume (\ref{assump1}) holds. Then $$
\mathcal{Z}\subseteq\bigoplus_{\mu\in I^\beta}\mathcal{P}_{\mu}e(\mu).
$$
\end{lem}

\begin{proof} For any $x\in\mathcal{Z}$ and $\mu\in I^\beta$, we have that $ze(\mu)=e(\mu)z$. It follows that $$
\mathcal{Z}\subseteq\bigoplus_{\mu\in I^\beta}e(\mu)\RR^\Lam(\beta)e(\mu).
$$
Now the assumption (\ref{assump1}) implies that $e(\mu)\RR^\Lam(\beta)e(\mu)=\mathcal{P}_{\mu}e(\mu)$. This proves the lemma.
\end{proof}

\begin{dfn} Let $\mu,\nu\in I^\beta$ with  $e(\mu)\neq 0\neq e(\nu)$ in $\RR^\Lam(\beta)$. We define $\mu\sim\nu$ if and only if there exists a sequence
$\mu^{[0]}:=\mu, \mu^{[1]},\cdots,\mu^{[k-1]},\mu^{[k]}:=\nu\in I^\beta$ such that for any $1\leq t\leq k$, $$
e(\mu^{[t-1]})\RR^\Lam(\beta)e(\mu^{[t]})\neq 0 .
$$
\end{dfn}
It is clear that ``$\sim$'' is an equivalence relation on $\{\mu\in I^\beta|\text{$e(\mu)\neq 0$ in $\RR^\Lam(\beta)$}\}$.

\begin{prop}\label{prop41} Assume (\ref{assump1}) holds. Let $\mu,\nu\in I^\beta$ with $e(\mu)\neq 0\neq e(\nu)$ in $\RR^\Lam(\beta)$. Then we have $\mu\sim\nu$.
\end{prop}

\begin{proof} If $\mu=\nu$, then the proposition follows from Lemma \ref{Lm41}. In this case we set $a(\mu,\nu)=n, d(\mu,\nu)=0$. Now suppose $\mu\neq\nu$. Let $1\leq a\leq n$ be the minimal integer such that
$\mu_{a}\neq\nu_{a}$. Note that $\mu,\nu\in I^\beta$ and (\ref{assump1}) holds. Thus we have that $\nu_{t_1}=\mu_a$, where $a<t_1\leq n$, and $\mu_s=\nu_s$ for any $1\leq s\leq a-1$. In this case, we set $a(\mu,\nu):=a$, $d(\mu,\nu):=t_1-a$.
We use induction downwards on $a(\mu,\nu)$ and upwards on $d(\mu,\nu)$ to prove the proposition.

Since $e(\nu)\neq 0$, we have $N^\Lam(1,\nu,k)>0$ for any $1\leq k\leq n$. By assumption, $\nu_{t_1}=\mu_{a}$ and $t_{1}>a$. Then $$\begin{aligned}
&\quad\,N^\Lam(s_{t_{1}-1},s_{t_{1}-1}\nu,t_1-1)=N^\Lam(1,\nu,t_1)+\<\alpha_{\nu_{t_1-1}},h_{\nu_{t_1}}\>\\
&=\<\Lam,h_{\nu_{t_1}}\>-\sum_{1\leq j<t_1-1}\<\alpha_{\nu_j},h_{\nu_{t_1}}\>\\
&\geq\<\Lam,h_{\mu_a}\>-\sum_{1\leq j<a}\<\alpha_{\nu_j},h_{\mu_a}\>
=N^\Lam(1,\mu,a)>0 .
\end{aligned}
$$
For any $1\leq k\leq n$ with $k\neq t_1-1, t_1$, one can check by definition that $$
N^\Lam(s_{t_{1}-1},s_{t_{1}-1}\nu,k)=N^\Lam(1,\nu,k)>0,
$$
while $N^\Lam(s_{t_{1}-1},s_{t_{1}-1}\nu,t_1)=N^\Lam(1,\nu,t_1-1)>0$. Applying Corollary \ref{cor42}, we can deduce that $$
e(s_{t_{1}-1}\nu)\RR^\Lam(\beta)e(\nu)\neq 0 .
$$
In particular, $s_{t_{1}-1}\nu\sim\nu$. Note that $(s_{t_{1}-1}\nu)_{t_1-1}=\nu_{t_1}=\mu_a$ and $(s_{t_{1}-1}\nu)_s=\mu_s$ for any $1\leq s\leq a$. We are in a position to apply the induction hypothesis, which implies that $s_{t_{1}-1}\nu\sim\mu$. Thus $\mu\sim\nu$. This completes the proof of the proposition.
\end{proof}

\medskip
\noindent
{\bf{Proof of Theorem \ref{mainthmC} Part 2)}:} Set $I_0:=\{\mu\in I^\beta|e(\mu)\neq 0\}$. In view of Lemma \ref{lem43}, it suffices to show that for any proper subset $J\subsetneq I_0$, $\sum_{\mu\in J}e(\mu)$ is not a center element in $\RR^\Lam(\beta)$.

Suppose this is not the case. Let $J\subsetneq I_0$ be a proper subset of $I_0$  such that $0\neq e_J:=\sum_{\mu\in J}e(\mu)\in\mathcal{Z}$. Applying Proposition \ref{prop41} we can find $\nu\in J$ and $\nu'\in I_0\setminus J$ such that $e(\nu)\RR^\Lam(\beta)e(\nu')\neq 0$. We fix a nonzero element
$0\neq x\in e(\nu)\RR^\Lam(\beta)e(\nu')$. Then as $e(\mu)e(\mu')=\delta_{\mu,\mu'}e(\mu)$, we have $$
0=xe_J=e_J x=e(\nu)x=x,
$$
which is a contradiction! This completes the proof of Theorem \ref{mainthmC}.
\qed\medskip

\begin{rem}
In \cite[Conjecture 3.33]{SVV}, Shan, Varagnolo and Vasserot have conjectured for the usual cyclotomic quiver Hecke algebra, that $\dim\mathcal{Z}_0=1$ and hence $\RR^\Lam(\beta)$ is indecomposable for any
symmetrizable Cartan matrix under the assumption $\cha\bk=0$, and they proved the conjecture when $\mathfrak{g}$ is symmetric of finite type and $\cha\bk=0$. Our above theorem verifies this conjecture for the more general cyclotomic quiver Hecke superalgebra and those special $\beta$ but without any assumption on $\cha\bk$. The argument makes essentially use of our dimension formula (Corollary \ref{smaedimcor}). We hope this approach can be generalized to work for any $\beta\in Q_n^+$.
\end{rem}

\end{document}